\documentclass[12pt]{article}
\usepackage{enumerate}
\usepackage{amsmath,epsfig,amssymb,amsbsy,verbatim,array,color,graphics,graphicx}
\usepackage{amssymb,amsmath,amsthm,amsfonts}
\usepackage{graphicx}
\usepackage{dsfont}

\newtheorem{question}{Question}
\newtheorem{corollary}[question]{Corollary}

\newtheorem{conjecture}[question]{Conjecture}
\newtheorem{theorem}[question]{Theorem}

\newtheorem{proposition}[question]{Proposition}
\newtheorem{lemma}[question]{Lemma}
\newtheorem{remark}[question]{Remark}

\newtheorem{claim}[question]{Claim}
\newtheorem{definition}[question]{Definition}

\numberwithin{question}{section}
\numberwithin{equation}{section}

\newcommand\Z{\mathbb{Z}}
\newcommand\E{\mathbb{E}}

\newcommand\Prb{\mathbb{P}}
\newcommand\cP{{\mathcal P}}

\newcommand\sfrac[2]{{\textstyle\frac{#1}{#2}}}

\title{Bootstrap percolation in random geometric graphs}
\author{Victor Falgas-Ravry\footnote{Ume{\aa} University, Ume{\aa}, Sweden. Email: \texttt{victor.falgas-ravry@umu.se}. Research supported by Swedish Research Council grant VR 2016-03488.} \and  Amites Sarkar\footnote{Western Washington University, Bellingham WA, USA. Email: \texttt{amites.sarkar@wwu.edu}.}}
\begin{document}
	\maketitle
	\setlength{\unitlength}{1in}
\begin{abstract}
Following Bradonji\'c and Saniee, we study a model of bootstrap percolation on the Gilbert random geometric graph on the $2$-dimensional  torus. In this model, the expected number of vertices of the graph is $n$, and the expected degree of a vertex is $a\log n$ for some fixed $a>1$. Each vertex is added with probability $p$ to a set $A_0$ of initially infected vertices. Vertices subsequently become infected if they have at least $ \theta a \log n $ infected neighbours.  Here $p, \theta \in [0,1]$ are taken to be fixed constants.

We show that if  $\theta < (1+p)/2$, then a sufficiently large local outbreak leads with high probability to the infection spreading globally, with all but $o(n)$ vertices eventually becoming infected. On the other hand, for $ \theta > (1+p)/2$, even if one adversarially infects every vertex inside a ball of radius $O(\sqrt{\log n} )$, with high probability the infection will spread to only $o(n)$ vertices beyond those that were initially infected.

In addition we give some bounds on the $(a, p, \theta)$ regions ensuring the emergence of large local outbreaks or the existence of islands of vertices that never become infected. We also give a complete picture of the (surprisingly complex) behaviour of the analogous $1$-dimensional bootstrap percolation model on the circle. Finally we raise a number of problems, and in particular make a conjecture on an `almost no percolation or almost full percolation' dichotomy which may be of independent interest.

\end{abstract}
	\section{Introduction}
	\subsection{Background}
	Bootstrap percolation encompasses a widely-studied family of cellular automata on networks. Originally introduced by Chalupa, Leath and Reich in 1979~\cite{ChalupaLeathReich79} in the context of magnetic systems, it has since been used to model a great variety of phenomena --- from the spreading of fads or beliefs in social networks~\cite{DreyerRoberts09, Granovetter78, RamosShaoReisAnteneodoAndradeHavlinMakse15}, to financial contagion and default on obligations in economic networks~\cite{AminiContMinca16}, to the activation of neurons in the brain~\cite{Amini10, Kozma07} or to the spread of viruses in human populations~\cite{DreyerRoberts09}. This plethora of applications has led to significant work on bootstrap percolation from network scientists, physicists, engineers and computer scientists as well as mathematicians.

	Formally an $r$-threshold bootstrap process on a graph $G=(V,E)$  is defined as follows. At the time $t=0$, an initial set of vertices $A_0\subseteq V$ is infected (or activated, if one prefers to avoid contagious connotations). Then at each time step $t\geq 0$, the vertices of $G$ having at least $r$ neighbours in $A_t$ become infected and are added to $A_t$ to form $A_{t+1}$. The infection thus spreads throughout the graph, and results in a set $A_{\infty}:=\bigcup_{t\geq 0} A_t$ of eventually infected vertices.

	If $G$ is a finite graph, the key question of interest is then: what proportion of vertices of $G$ eventually become infected? This obviously depends on the choice of the set of initially infected vertices $A_0$. In most of the work on bootstrap percolation to date, $A_0$ is chosen according to a Bernoulli process on the vertices of $G$: each vertex $v\in V$ is included in $A_0$ with probability $p$ independently of all other vertices. One then asks for which initial infection probability $p$ does the infection spread to `most' of $V$ (\emph{almost percolation}, $\vert A_{\infty}\vert =\vert V\vert(1+o(1))$), or to all of $V$ (\emph{percolation}, $\vert A_{\infty}\vert =\vert V\vert$) with high probability (\emph{w.h.p.}, meaning with probability $1-o(1)$).

	Aizenman and Lebowitz~\cite{AizenmanLebowitz88} were the first to investigate this kind of question when $G=[n]^d$, the $d$-dimensional $n \times n \times \cdots \times n$ grid graph. In a landmark result in 2003, Holroyd~\cite{Holroyd03} showed that for $2$-threshold bootstrap percolation on $[n]^2$, $p_c([n]^2, 2)=\frac{\pi^2}{18 \log n}$ was a sharp threshold for percolation, in the sense that if $p = \frac{c}{\log n}$ is the initial infection probability and $c>0$ is fixed, then if $c<\frac{\pi^2}{18}$  w.h.p.\ the infection does not spread to the entire square grid, while if $c>\frac{\pi^2}{18}$ then w.h.p.\ every vertex of $[n]^2$ eventually becomes infected. Surprisingly, this result disproved predictions for the value of the critical threshold that had been made based on numerical simulations for the problem. Holroyd's results were then extended to other dimensions $d$ and thresholds $r$  by Balogh, Bollob\'as, Duminil-Copin and Morris~\cite{BaloghBollobasDuminilCopinMorris12}.

	 Motivated by the applications of bootstrap percolation to the modelling of real-life network phenomena, there has been growing interest in the past decade for the study of bootstrap percolation on random graphs. An important work in this vein was the study of bootstrap percolation on Erd{\H o}s--R\'enyi random graphs $G_{n,q}$ by Janson, {\L}uczak, Turova and Vallier~\cite{JansonLuczakTurovaVallier12} in 2012. Recall that the random graph $G_{n,q}$ is obtained by taking $[n]:=\{1,2, \ldots n\}$ as a vertex-set, and including each pair $ij$ as an edge of the random graph with probability $q$, independently of all the other pairs. The authors of~\cite{JansonLuczakTurovaVallier12} determined inter alia for every threshold $r\geq 2$ and $q$ the critical thresholds $p=p(n,r,q)$ for the infection probability at which the size of $A_{\infty}$ goes  from w.h.p.\ $o(n)$ to w.h.p.\ $n-o(n)$ (almost percolation) and from w.h.p.\ $n-o(n)$ to w.h.p.\ $n$ (percolation).

	 Bootstrap percolation has been rigorously studied on several other random graph models: random regular graphs~\cite{BaloghPittel07, Janson09}, power-law random graphs~\cite{AminiFountoulakis14}, Bienaym\'e--Galton--Watson trees~\cite{BollobasGundersonHolmgrenJansonPrzykucki14}, random graphs with specified vertex degrees~\cite{Janson09}, toroidal grids with random long edges added in (a special case of the Kleinberg model)~\cite{JansonKozmaRuszinkoSokolov19}, inhomogeneous random graphs~\cite{FountoulakisKangKochMakai18}, amongst others.  A motivation for the latter two models is that they may have degree distributions or spatial characteristics that more closely resemble those of the real-life networks motivating the study of bootstrap percolation.

	 For similar reasons, there has been interest in bootstrap percolation models on random \emph{geometric} graphs. Indeed, many real-life networks have a distinctly spatial structure that affects their behaviour and properties. Thus it is natural to study bootstrap percolation on random graph models in which geometry plays a role. This was first done by Bradonji\'c and Saniee~\cite{BradonjicSaniee14} in 2014, who introduced a model for bootstrap percolation on random geometric graphs that is the focus of the present paper.

	 The Gilbert disc model is the most widely-studied random geometric graph model, and was first defined by Gilbert~\cite{Gilbert61} in 1961; indeed `random geometric graph' without any further qualifier usually refers to the Gilbert model. Given a measurable metric space $\Omega$, a Gilbert random geometric graph $G_r(\Omega)$ is obtained by taking as the vertex-set the point-set $\mathcal{P}$ resulting from a Poisson point process  of intensity $1$ on $\Omega$. Two vertices $v,v'$ in $\mathcal{P}$ are then joined by an undirected edge if their distance in $\Omega$ is less than $r$. In other words, the neighbours of a vertex $v$ are precisely those points of $\mathcal{P}\setminus\{v\}$ that lie inside the ball of radius $r$ centred at $v$.

	 Gilbert studied his model in the case where $\Omega=\mathbb{R}^2$, the $2$-dimensional plane equipped with the usual Euclidean distance and Lebesgue measure. The study of the Gilbert model and related random geometric graph models on such unbounded spaces is known as \emph{continuum percolation}, and is the subject of a monograph of Meester and Roy~\cite{MeesterRoy96}. In a different direction, researchers have been interested in random geometric graph models in bounded, finite-dimensional spaces, in particular when $\Omega$ is either $S_n^d:=[0,n^{\frac{1}{d}}]^d$, the $d$-dimensional box of volume $n$, or $T_n^d:=\left(\mathbb{R}/n^{1/d}\mathbb{Z}\right)^d$, the $d$-dimensional torus of volume $n$, where $d$ is fixed and $n$ is large. For both of these choices of $\Omega$, standard results on concentration of the Poisson distribution imply that $G_r(\Omega)$ w.h.p.\  has $n+o(n)$ vertices.

	 In the case $\Omega=T_n^d$ (where we can ignore boundary effects and all vertices look the same), $G_r(\Omega)$ can be viewed as a natural geometric analogue of the Erd{\H o}s--R\'enyi random graph. Let $\alpha_d$ denote the volume of the $d$-dimensional unit ball. Then the expected degree of a vertex in $G_r(T_n^d)$ is precisely $\alpha_d r^d$. A classical result of Penrose~\cite{Penrose97} established that the threshold for connectivity for $G_r(T_n^d)$  occurs at   $\alpha_d r^d=\log n$: if $\alpha_d r^d=a\log n$ and $a>0$ is fixed, then for $a<1$ w.h.p.\  $G_r(T_n^d)$ contains isolated vertices and thus fails to be connected, while for $a>1$ w.h.p.\ $G_r(T_n^d)$ is connected. Much more is known about Gilbert random geometric graphs (which, together with the closely related $k$-nearest neighbour model, have been applied in a variety of contexts, for example to model sensor networks~\cite{BalisterBollobasSarkarKumar07} and  ad hoc wireless networks~\cite{XueKumar04}, and for cluster analysis in spatial statistics~\cite{GonzalesBarriosQuiroz03}), and we refer an interested reader to the monograph of Penrose~\cite{Penrose03} devoted to the topic.

	 Bradonji\'c and Saniee considered the Gilbert disc model $G_r(T_n^2)$ where $r$ is given by $\pi r^2=a\log n$, for some constant $a>1$ and $n$ is large. They studied $\theta a \log n$-threshold bootstrap percolation on this host graph --- i.e. where the threshold for infection is a proportion $\theta$ of the expected degree of a vertex. This is somewhat in contrast to previous work, where typically the threshold for infection was fixed rather than growing with the number of vertices $n$, but may be a more suitable choice of parameter for modelling situations such as the spread of a fad or fashion in a social network.

	  Bradonji\'c and Saniee's paper featured a mixture of rigorous results and simulations. On the theoretical side, they proved two results. First of all, they determined~\cite[Theorem 1]{BradonjicSaniee14} an explicit function $f_1(a, \theta)$ such that if the initial infection probability $p$ is fixed and satisfies $p<f_1(a, \theta)$, then w.h.p.\ the infection does not spread at all: $A_{\infty}=A_0$, and every vertex that is initially uninfected stays uninfected forever. Secondly, Bradonji\'c and Saniee determined~\cite[Theorem 2]{BradonjicSaniee14} a second function $f_2(a, \theta)$ such that if $p$ is fixed and satisfies $p>f_2(a, \theta)$, then w.h.p.\ there is full percolation: every vertex becomes infected. The function $f_1(a,\theta)$ in the first of these results is easily seen to be best possible (see Proposition~\ref{prop: start} below). However the bound $f_2(a, \theta)$ in the second result seems far from optimal --- indeed, the simulations performed by Bradonji\'c and Saniee suggest as much.

	 Besides Bradonji\'c and Saniee's 2014 paper, comparatively little mathematical work appears to have been done on bootstrap percolation in random geometric graphs. In a 2016 work Candellero and Fountoulakis~\cite{CandelleroFountoulakis16} studied $r$-threshold bootstrap percolation on hyperbolic random geometric graphs for constant $r$, and determined for their model a critical probability $p_c$ such that if the initial infection probability $p$ satisfies $p\ll p_c$ then w.h.p.\ the infection does not spread at all, while if $p\gg p_c $ then w.h.p.\ the infection spreads to a strictly positive proportion of the vertices. More recently, Koch and Lengler~\cite{KochLengler16, KochLengler21} studied a localised form of $r$-threshold bootstrap percolation on geometric inhomogeneous random graphs, with $r$ a fixed constant and  where the set of initially infected vertices is located within some bounded source region $B$ (rather than the whole space), and determined a similar critical threshold $p_c$ below which w.h.p.\ an infection does not spread at all, and above which w.h.p.\ an infection spreads to a positive proportion of all vertices. Finally in a very recent PhD thesis, Whittemore~\cite{WhittemoreThesis} studied   bootstrap percolation in the Gilbert random geometric graph when the infection threshold  is constant, and determined amongst other things the thresholds at which the model's typical behaviour transitions from almost no percolation to almost percolation.  As far as we are aware, this is the (surprisingly limited) extent of rigorous mathematical study of bootstrap percolation on random geometric graph models (though there also exist some experimental and simulation results for bootstrap percolation on geometric scale-free networks, see e.g.~\cite{GaoZhouHu15}).

\subsection{The Bradonji\'c--Saniee  model}
For the reader's convenience, we restate here the precise model we shall be studying in this paper.

Let $T_n^d:=\left(\mathbb{R}/n^{1/d}\mathbb{Z}\right)^d$ denote the $d$-dimensional torus of hypervolume $n$. Given a parameter $r$, a Gilbert random geometric graph $G_r(T_n^d)$  on $T_n^d$ is obtained as follows: we let its vertex set $\cP$ be the result of a Poisson point process of intensity $1$ on $T_n^d$, and join vertices $u,v\in \cP$ by an edge if their distance (in the torus) is less than $r$. For compactness of notation, we use $G_{n,r}^d$ to denote $G_r(T_n^d)$.

Let $a>1$ and $d\in \mathbb{N}$ be fixed. Let $r$ be given by the relation $\alpha_d r^d=a\log n$, where $\alpha_d$ is the volume of the $d$-dimensional unit ball. Note that for this choice of parameters the expected total number of vertices in $G_{n,r}^d$ is $n$, while the expected degree of a vertex in $G_{n,r}^d$ is $a\log n$. Further, by classical results of Penrose~\cite{Penrose97}, $G_{n,r}^d$ is w.h.p.\ connected.

Bradonji\'c and Saniee~\cite{BradonjicSaniee14} introduced the following model of bootstrap percolation on $G_{n,r}^d$: let $p, \theta\in [0,1]$ be fixed. At time $t=0$, let each vertex of $G_{n,r}^d$ be infected independently at random with probability $p$. Denote by $A_0$ this set of initially infected vertices. The infection then spreads through the graph $G_{n,r}^d$ as follows: at each time step $t>0$, all vertices of $G_{n,r}^d$ which have at least $\theta a \log n$ infected neighbours (i.e.\ neighbours in the infected set $A_{t-1}$) become infected themselves and are added to $A_{t-1}$ to form the set $A_t$. We denote by $A_{\infty}=\bigcup\{A_t:\ t\in \mathbb{Z}_{\geq 0}\}$ the set of all vertices of $G_{n,r}^d$ that eventually become infected under this process.

With $a, p, \theta, d$ fixed the main question of interest in this model is: what is the typical size of $A_{\infty}$ for large $n$? In this paper we investigate this question in detail in dimensions $d=1$ and $d=2$.
\subsection{Contributions of this paper}
Our main contribution in this paper is identifying in dimension $d=2$ the threshold at which a sufficiently large \emph{local} outbreak can cascade and lead to a \emph{global infection}. Say that a ball $B$ in $T_n^2$ is \emph{infected} if all vertices of $G_{n,r}^2$ that lie inside $B$ are infected.
\begin{theorem}\label{theorem: (1+p)/2 threshold}
Let $a, \theta, p$ be fixed. Then the following hold.
\begin{enumerate}[(i)]
	\item If $\theta<\frac{1+p}{2}$, then there exists a constant $C=C(a, \theta, p)$ such that w.h.p.\ if any ball $B$ in $T_n^2$ of radius $Cr$ is infected (either artificially or as a result of the bootstrap percolation process), then all but $o(n)$ vertices of $G_{n,r}^2$ eventually become infected. Furthermore, when the infection stops, all connected components of uninfected vertices in $G_{n,r}^2[\mathcal{P}\setminus A_{\infty}]$ have Euclidean diameter $O(\sqrt{\log n})$ in $T_n^2$.
	\item If $\theta>\frac{1+p}{2}$, then for every constant $C>0$, w.h.p.\ even if one adversarially selects a ball $B$ in $T_n^2$ of radius $Cr$ and infects all the vertices it contains, only $o(n)$ additional vertices of $G_{n,r}^2$ become infected in the bootstrap percolation process starting from the initially infected set $A_0\cup (B\cap \mathcal{P})$.  What is more, all connected components of $G_{n,r}^2[A_{\infty}\setminus \left(A_0\cup B\right)]$ have Euclidean diameter $O(\sqrt{\log n})$ in $T_n^2$.
\end{enumerate}
\end{theorem}
\noindent Note that in our regime $\pi r^2=a\log n$ and  the longest edge of $G_{n,r}^2$ thus has length $O(\sqrt{\log n})$. We therefore view point-sets of diameter $O(\sqrt{\log n})$ as \emph{local} configurations. What Theorem~\ref{theorem: (1+p)/2 threshold} says is thus that for $\theta<\frac{1+p}{2}$, a sufficiently large local infection will, with the help of the initially infected vertices, spread to most of the graph, leaving only isolated local `islands' of uninfected vertices, while for $\theta>\frac{1+p}{2}$, all infectious local outbreaks remain local.

This leads us to conjecture that in the Bradonji\'c--Saniee model on $G_{n,r}^2$ with $\theta\neq \frac{1+p}{2}$, w.h.p.\ either an initial infection spreads to at most $o(n)$ new vertices (almost no percolation), or it spreads to all but at most $o(n)$ vertices (almost percolation).
\begin{conjecture}[Almost no percolation/almost full percolation dichotomy]\label{conjecture: either almost no percolation or almost percolation}
	Let $(a, p, \theta)$ be fixed with $\theta\neq \frac{1+p}{2}$ and $a>1$. Then in the Bradonji\'c--Saniee model for bootstrap percolation on $G_{n,r}^2$, w.h.p.\ either $\vert A_{\infty}\setminus A_0\vert =o(n)$ or $\vert \mathcal{P}\setminus A_{\infty}\vert =o(n)$.
\end{conjecture}
\noindent We were unfortunately unable to resolve Conjecture~\ref{conjecture: either almost no percolation or almost percolation} in full, but our results imply it holds if $\theta>\frac{1+p}{2}$ or $\theta<\theta_{\mathrm{local}} $, where $\theta_{\mathrm{local}} =\theta_{\mathrm{local}}(a,p)$ is a quantity arising as the solution to an explicit continuous optimisation problem and whose technical definition (Definition~\ref{def: theta sym local growth}) we defer to Section~\ref{subsection: local growth}. Suffice it to say here that $\theta_{\mathrm{local}}$ is the threshold for the appearance of large local, symmetrically distributed, infectious outbreaks.
\begin{theorem}\label{theorem: threshold for symmetric local growth}
	Let $(a,p, \theta)$ be fixed with $a>1$ and
	\begin{align}\label{eq: symmetric growth threshold}
	\theta< \theta_{\mathrm{local}}(a,p).
	\end{align}
	Then w.h.p.\ almost percolation occurs in the Bradonji\'c--Saniee model, i.e.\ $\vert \mathcal{P}\setminus A_\infty\vert =o(n)$. 	
\end{theorem}
\noindent We note here the fact that the `symmetric local growth condition' $\theta < \theta_{\mathrm{local}}(a,p)$ implies the `global growth condition' $\theta <(1+p)/2$ (see Proposition~\ref{proposition: inequality for theta local}): $\theta_{\mathrm{local}}\leq (1+p)/2$ for all fixed $a>1$ and $p\in [0,1]$. 
We believe that $\theta_{\mathrm{local}}$ gives the threshold for almost percolation in the  Bradonji\'c--Saniee model for bootstrap percolation, and thus that  the following strengthening of Conjecture~\ref{conjecture: either almost no percolation or almost percolation} is true.
\begin{conjecture}\label{conjecture: symmetric growth}[Symmetric local growth]
Let $(a,p, \theta)$ be fixed with $a>1$ and
	\begin{align*}
	\theta> \theta_{\mathrm{local}}(a,p).
	\end{align*}
	Then w.h.p.\ almost no percolation occurs in the Bradonji\'c--Saniee model, i.e.\ $\vert \mathcal{A}_{\infty}\setminus A_0\vert =o(n)$.
\end{conjecture}
The content of Conjecture~\ref{conjecture: symmetric growth} is two-fold: the conjecture asserts first of all that completely infecting a large local ball is w.h.p.\ necessary for the infection to spread globally, and secondly that the likeliest way an infection spreads to a large local ball is if there is an abnormally high concentration of initially infected and of initially non-infected vertices distributed in a symmetric manner around the centre of the said ball.

Theorems~\ref{theorem: (1+p)/2 threshold} and~\ref{theorem: threshold for symmetric local growth} above are stated and proved for the Bradonji\'c--Saniee bootstrap percolation model in the torus $T_n^2$ rather than the square $S_n^2:=[0,\sqrt{n}]^2$ to avoid technical complications due to boundary effects. However, as we note in Section~\ref{section: concluding remarks}, our results also hold for their model in the square, modulo a technical modification in the statement of Theorem~\ref{theorem: (1+p)/2 threshold}(i). We thus expect Conjectures~\ref{conjecture: either almost no percolation or almost percolation} and~\ref{conjecture: symmetric growth} to also hold in the square.

Given Theorem~\ref{theorem: threshold for symmetric local growth} and Conjecture~\ref{conjecture: symmetric growth} on almost percolation, it is natural to ask how much smaller $\theta$ needs to be to ensure full percolation: which triples $(a,p, \theta)$ guarantee that a global infection w.h.p.\ infects every vertex of $G_{n,r}^2$?  We are unable to answer this question exactly. However, as in Theorem~\ref{theorem: threshold for symmetric local growth} we are able to determine the threshold $\theta_{\mathrm{islands}}$ for the disappearance of certain symmetric `islands' of uninfected vertices, which provide what we conjecture is the main obstacle to full percolation. Here $\theta_{\mathrm{islands}}=\theta_{\mathrm{islands}}(a,p)$ is an (explicit) solution to a certain optimisation problem, whose formal definition we defer to Section~\ref{subsection: islands}. (Note that $\theta_{\mathrm{islands}}$ will satisfy the inequality $\theta_{\mathrm{islands}}\leq \frac{1+p}{2}$, see~\eqref{inequality for theta island}.)
\begin{theorem}\label{theorem: islands}
Let $(a,p, \theta)$ be fixed with $a>1$ and $\theta>\theta_{\mathrm{islands}}(a,p)$. Then w.h.p.\ some vertices remain uninfected by the end of the bootstrap percolation process in the Bradonji\'c--Saniee model in the torus $T_n^2$, i.e.\ $\vert \mathcal{P}\setminus A_{\infty}\vert>0$ and we do not have full percolation.
\end{theorem}
\begin{conjecture}\label{conjecture: islands}
Let $(a,p, \theta)$ be fixed with $a>1$ and $\theta<\min\left(\theta_{\mathrm{islands}}, \theta_{\mathrm{local}}\right)$.  Then w.h.p.\  we have full percolation in the Bradonji\'c--Saniee model in the torus $T_n^2$, i.e.\ $\mathcal{P}= A_{\infty}$.
\end{conjecture}
\noindent Theorem~\ref{theorem: islands} carries over immediately to the square setting, but Conjecture~\ref{conjecture: islands} does not: in the square setting, one will need to separately compute the threshold for the disappearance of uninfected islands close to the boundary, which will require additional calculations (this is a rather standard feature in results on random geometric graphs in the square; see e.g.\ the proof of~\cite[Theorem 7]{BalisterBollobasSarkarWalters}); not such liminal islands will have fewer neighbouring vertices and will be harder to infect.

For the Bradonji\'c--Saniee bootstrap percolation model in general dimension $d\geq 1$, we also determine the threshold $\theta_{\mathrm{start}}(a,p)$ for the event $A_0\neq A_{\infty}$ to hold w.h.p.\ (Proposition~\ref{prop: start}), generalising~\cite[Theorem 1]{BradonjicSaniee14}, and for the event that there exist `uninfectable' initially uninfected vertices with degree strictly less than $\theta a \log n$, giving lower bounds on the threshold for full percolation $f_2(a, \theta)$ given in~\cite[Theorem 2]{BradonjicSaniee14}.

Together with the results and conjectures above, these last results give the following picture for the expected behaviour of the Bradonji\'c--Saniee model in dimension $2$ with $(a,p, \theta)$ fixed and $a>1$ (see Figure 1):
\begin{itemize}
	\item for $\theta>\theta_{\mathrm{start}}$ w.h.p.\ we have no percolation: $A_0=A_{\infty}$;
	\item for $\theta_{\mathrm{start}}>\theta > \frac{1+p}{2}$, w.h.p.\ we have almost no percolation, $\vert A_{\infty}\setminus A_0\vert =o(n)$, even if we adversarially infect a ball of area $O(\log n)$;
	\item for $ \frac{1+p}{2}> \theta >\theta_{\mathrm{local}}$, w.h.p.\ we have almost no percolation, but adversarially infecting a ball of area $\Omega(\log n)$ w.h.p.\ leads to almost percolation;
	\item for $\theta_{\mathrm{local}}> \theta > \theta_{\mathrm{islands}}$, w.h.p.\ we have almost percolation but not full percolation, $0<\vert \mathcal{P}\setminus A_{\infty}\vert \leq o(n)$;
	\item for $\min \left(\theta_{\mathrm{local}}, \theta_{\mathrm{islands}}\right)>\theta $, w.h.p.\ we have full percolation, $ \mathcal{P}= A_{\infty}$.
\end{itemize}


\begin{figure}[htp!]
\centering
\includegraphics[width=0.8\linewidth]{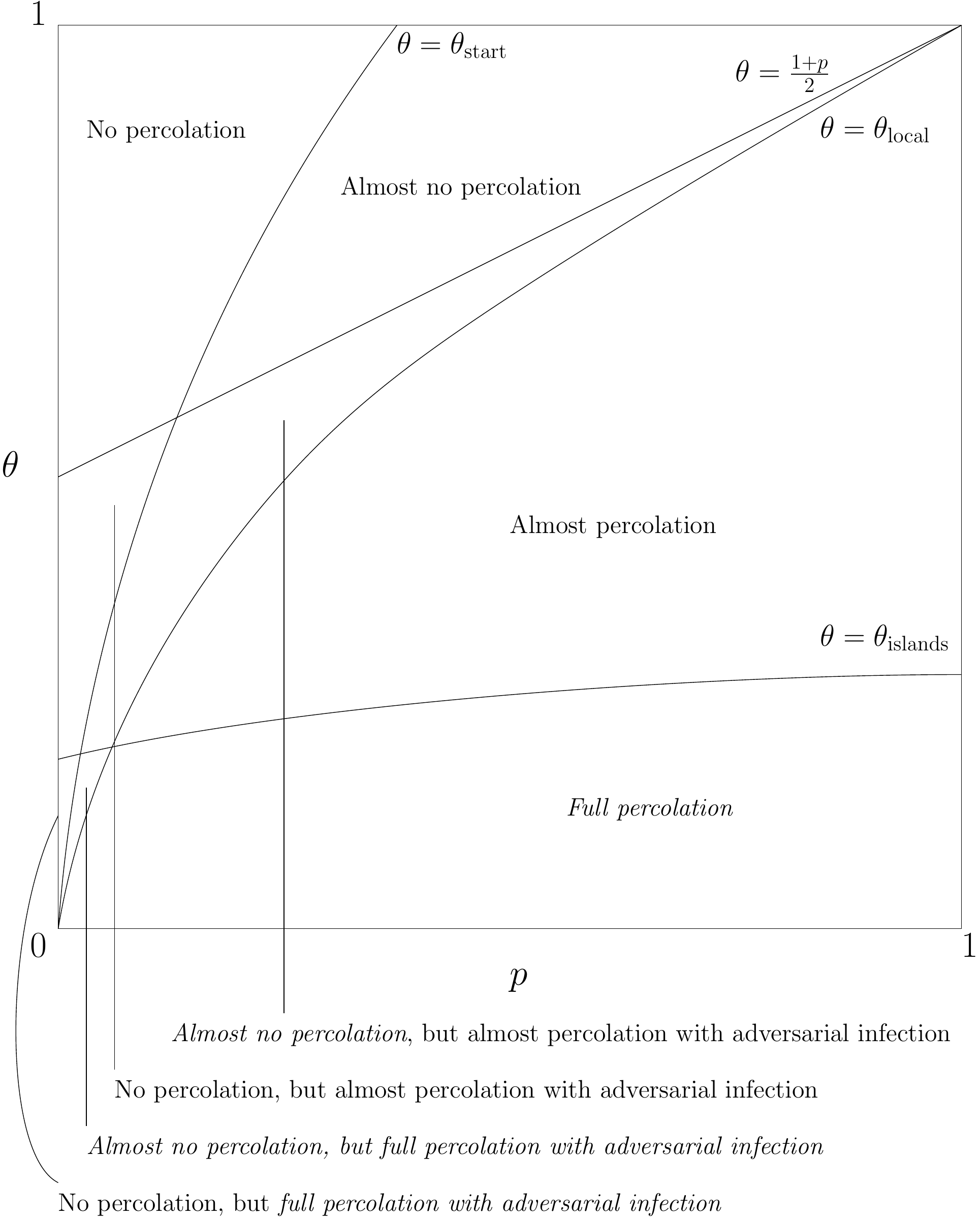}
\caption{\small Phase diagram in the $(p,\theta)$-plane for $a>1$ fixed; italics indicate conjectured behaviour}
\label{phase_d=2}
\end{figure}

Finally, we give a complete picture of the typical behaviour of the Bradonji\'c--Saniee bootstrap percolation model on the graph $G_{n,r}^1$, i.e.\ in the $1$-dimensional case. This turns out to be surprisingly complex, involving a $(p, \theta)$-phase diagram with six different regions (see the summary in Section~\ref{section: 1d}); however, unlike in the $2$-dimensional case, we can compute the various thresholds more or less explicitly. We defer an exact statement of these $1$-dimensional results to Section~\ref{section: 1d}.

\subsection{Organisation of the paper}
The remainder of this paper is organised as follows. In Section~\ref{section: preliminaries}, we prove some basic probabilistic results for Poisson point processes required for later results. We also derive thresholds for the events that some initially infected vertices eventually become infected (Proposition~\ref{prop: start}), and that some initially uninfected vertices have degree too low to ever become infected (Proposition~\ref{proposition: 0-stop}).

In Section~\ref{section: 2d}, we prove our main results for bootstrap percolation on the $2$-dimensional torus, while in Section~\ref{section: 1d}, we outline the behaviour of the Bradonji\'c-Saniee model on the circle. We end the paper in Section~\ref{section: concluding remarks} with a discussion of some of the many open problems on bootstrap percolation for random geometric graphs.

\section{Preliminaries}\label{section: preliminaries}
\subsection{Notation}
Given a host metric space $\Omega$ and a point $\mathbf{x}\in \Omega$, we write $B_r(\mathbf{x})$ for the ball in $\Omega$ of radius $r$ centred at $\mathbf{x}$. We use $\vert S\vert$ to denote the size of $S$ if $S$ is a finite set, and the Lebesgue measure of $S$ otherwise. Given $\mathbf{x}\in \mathbb{R}^d$, we denote by $\|\mathbf{x}\|$ the standard Euclidean $\ell_2$-norm of $\mathbf{x}$.

We use standard graph theoretic terminology and Landau Big O notation throughout the paper. In particular, given a graph $H$ and a subset of its vertex set $X$, we denote by $H[X]$ the subgraph of $H$ induced by $X$ and by $H-X$ the subgraph of $H$ induced by $V(H)\setminus X$.
\subsection{Probabilistic tools}
The following lemma, taken from~\cite[Lemma 1]{BalisterBollobasSarkarWalters}, greatly facilitates calculations of event probabilities, and we shall make extensive use of it throughout the paper.
\begin{lemma}[Estimating probabilities for Poisson point processes]\label{lemma: Paul}
	Let $A\subset R^d$ be measurable, and let $\rho\ge0$ be a real number such that $\rho|A|\in\Z$.
	Then the probability that a Poisson process in $\mathbb{R}^d$ with intensity $1$ has precisely $\rho|A|$ points in the region $A$ is given by
	\[
	\exp\left\{(\rho-1-\rho\log\rho)|A|+O(\log_+\rho|A|)\right\}
	\]
	with the convention that $0\log 0=0$, and $\log_+ x=\max(\log x,1)$.
\end{lemma}
Further by standard properties of Poisson point processes (see e.g.\ Kingman~\cite{Kingman93}), note that the following are equivalent:
\begin{itemize}
	\item taking as the vertex-set of $G_{n,r}^d$ the outcome $\mathcal{P}$ of a Poisson point process of intensity $1$ on $T_n^d$, and then infecting each vertex of $\mathcal{P}$ independently at random with probability $p\in (0,1)$ to obtain $A_{0}$;
	\item letting $A_0$ be the outcome of a Poisson point  process of intensity $p$ on $T_{n}^d$, and taking as the vertex-set of $G_{n,r}^d$ the union $\mathcal{P}$ of $A_0$ and of the outcome of a Poisson point process of intensity $1-p$ on $T_{n}^d$.
\end{itemize}
We will thus be able to jump back and forth between these two equivalent ways of constructing $\mathcal{P}$, adopting whichever point of view makes our calculations simplest.

At several points in the paper, we will need to count the number of copies of a `local' event inside $T_{n}^d$, and the next lemma gives us some simple tools to do this. Let $\Gamma =\mathbb{Z}^d\cap [0, n^{1/d}]^d$; this may be identified with a $d$-dimensional grid of points inside $T^d_n$. Given an event $E$ defined for point-sets  in $T^d_n$, for every $\mathbf{x}\in \Gamma$ we let $E(\mathbf{x})$ denote the collection of point-sets in $T^d_n$  whose translates by $\mathbf{x}$ belong to $E$; in other words, a configuration of initially uninfected and initially infected points $(\mathcal{P}\setminus A_0, A_0)$ belongs to $E(\mathbf{x})$ if and only if $(\mathcal{P}\setminus A_0-\mathbf{x}, A_0-\mathbf{x})$ belongs to $E$.

We say that an event $E$ is \emph{$Cr$-bounded} if it is determined by what happens (i.e.\ which points are present/initially infected) within a ball of radius $Cr$ of the origin $\mathbf{0}$, for come constant $C>0$. For such an event, we define $\mathbf{E}=(\mathds{1}_{E(\mathbf{x})})_{\mathbf{x}\in \Gamma}$ to be the $\Gamma$-dimensional zero--one vector recording for which $\mathbf{x}\in \Gamma$ the event $E(\mathbf{X})$ occurs.
Given a $Cr$-bounded event $E$, let $N_E:=\sum_{\mathbf{x}}\mathds{1}_{E(\mathbf{x})}$ denote the number of $\mathbf{x}\in \Gamma$ for which $E(\mathbf{x})$ occurs.
\begin{lemma}\label{lemma: first and second moment}
	Suppose $E$ is a $Cr$-bounded event with $q:=\Prb(E)= e^{-c\log n +o(\log n)}$ for some fixed constant $c$. The following hold:
	\begin{enumerate}[(i)]
\item if	$c>1$, then with probability $1-o(1)$ we have $N_E=0$;
\item if $c<1$, then with probability $1-o(1)$ we have $N_E=n^{1-c+o(1)}$.		
\end{enumerate}
\end{lemma}
\begin{proof}
Part (i) is immediate from Markov's inequality.	 For the lower bound in part (ii), note that there exists a subset  $\Gamma'$ of $\Gamma$ satisfying
\begin{enumerate}[(a)]
	\item $\vert \Gamma' \vert =\Omega\left(\frac{n}{\log n}\right)$;
	\item $\forall \mathbf{x}, \mathbf{y}\in \Gamma'$ with $\mathbf{x}\neq \mathbf{y}$, $\|\mathbf{x}-\mathbf{y}\|> 2Cr$.
\end{enumerate}
Indeed, one can construct such a set $\Gamma'$ by greedily adding vertices from $\Gamma$ one by one subject to (b). By (b) and $Cr$-boundedness, the events $\left(E(\mathbf{x})\right)_{\mathbf{x}\in \Gamma'}$ are independent. Thus $N_E$ stochastically dominates a $\mathrm{Binomial}(\vert \Gamma'\vert, q)$ random variable. Since the expectation of this binomial random variable is $\vert \Gamma'\vert q =\Omega(nq/\log n)= n^{1-c+o(1)}\gg 1$, a standard Chernoff bound tells us that with probability $1-o(1)$,  $N_E\geq n^{1-c+o(1)}$.

For the upper bound in part (ii) we use Markov's inequality: the probability that $N_E$ is greater than $\left(\mathbb{E}N_E\right)\log n =O(nq \log n)= n^{1-c+o(1)}$ is $O(1/\log n)=o(1)$. Thus with probability $1-o(1)$, $N_E\leq n^{1-c+o(1)}$, as required.
\end{proof}

\subsection{Elementary considerations on the Bradonji\'c--Saniee model}
Consider the Bradonji\'c--Saniee bootstrap percolation model in dimension $d\geq 1$ with $\pi r^2=a\log n$ and $a>1$ fixed. When $\theta<p$ and $n$ is large, most vertices will immediately see more than $a\theta\log n<ap\log n$ infected neighbours, so that most uninfected vertices will immediately become infected. However, this is not the whole story. Indeed, even when $\theta$ is much smaller than $p$, there is still a chance that some vertex somewhere might see far fewer than its expected $a\log n$ neighbours (infected or not); if it in fact sees fewer than $\theta a\log n$ neighbours, then it can never become infected. In the other direction even when $\theta>p$, there could still be a chance that some vertex somewhere will see far more than its expected $pa\log n$ infected neighbours, perhaps as many as $\theta a\log n$, so that the infection could spread to that vertex.

Roughly speaking, an analysis of the Bradonji\'c--Saniee model must grapple with three separate questions: whether the infection starts to spread at all, whether it continues to spread to most of the graph,
and whether it finally infects every vertex. Each of these requires a separate analysis, even in one dimension.
As a simple consequence of the probabilistic tools we have introduced, however, we can readily answer here the question of when the infection starts at all, strengthening and generalising~\cite[Theorem 1]{BradonjicSaniee14}.
\begin{proposition}\label{prop: start}
	For $a>1$ and $0<p<\theta<1$ fixed, let
	\[
	f_{\rm start}(a,p,\theta):=a(p-\theta+\theta\log(\theta/p)).
	\]
	Then, if $f_{\rm start}(a,p,\theta)<1$, w.h.p.\  at least one initially uninfected vertex is infected in the first round
	of the bootstrap percolation process. If, however, $f_{\rm start}(a,p,\theta)>1$, then w.h.p.\ no initially uninfected vertex ever
	becomes infected.
\end{proposition}
\begin{proof}
	Write $X$ for the number of vertices which initially see more than $a\theta\log n$ infected neighbours in $G_{n,r}^d$. By Wald's identity and Lemma~\ref{lemma: Paul} we have for $0<p<\theta<1$ that,
	\begin{align*}
	\E(X)&=n\exp\Bigl\{pa\log n\left((\theta/p)-1-(\theta/p)\log(\theta/p)\right)+O(\log_+a\theta\log n)\Bigr\}\\
	&=\exp\Bigl\{\log n(1-f_{\rm start}(a,p,\theta)+o(1))\Bigr\}.
	\end{align*}
	Consequently, if $f_{\rm start}(a,p,\theta)>1$, then by Markov's inequality w.h.p.\ $X=0$ and no initially uninfected vertex of $\mathcal{P}$ ever becomes infected.

	If on the other hand $f_{\rm start}(a,p,\theta)<1$, then let $E$ denote the event that the ball of radius $r-\sqrt{d}$ around the origin contains at least  $\theta a \log n$ initially infected vertices and that the ball of radius $\sqrt{d}$ around the origin contains at least one initially uninfected vertex of $\mathcal{P}$. Then by Lemma~\ref{lemma: Paul} and standard properties of the Poisson point process, the probability of $E$ is
	\begin{align*}
\mathbb{P}(E)&=\exp\left\{-f_{\rm start}(a,p,\theta)\log n+o(\log n)\right\} \left(1-\exp\{-(1-p)\alpha_d (\sqrt{d})^d \}\right)\\
&= \exp\left\{-f_{\rm start}(a,p,\theta)\log n+o(\log n)\right\}.
\end{align*}
By Lemma~\ref{lemma: first and second moment}(ii) w.h.p.\ $N_E=n^{1-f_{\rm start}+o(1)}$. In particular w.h.p.\ the event $E(\mathbf{x})$ occurs for some $\mathbf{x}\in \Gamma$. This implies there exists an initially uninfected vertex $v\in B_{\sqrt{d}}(\mathbf{x})\cap \mathcal{P}$ such that $B_r(v)\cap A_0\supseteq B_{r-\sqrt{d}}(\mathbf{x})\cap A_0$ contains at least $\theta a \log n$ initially infected points of $\mathcal{P}$. Thus $v$ becomes infected in the first round of the bootstrap percolation process, and w.h.p.\ $A_0\neq A_1\subseteq A_{\infty}$. This concludes the proof of the proposition.
\end{proof}
\begin{definition}\label{def: theta start}
	For $a>1$ and $p\in (0,1)$ fixed, we define $\theta_{\mathrm{start}}= \theta_{\mathrm{start}}(a,p)$ to be the supremum of the $\theta \leq 1$ such that $f_{\rm start}(a,p,\theta)<1$.
\end{definition}
\noindent Since $f_{\rm start}(a,p,p)=0$, it follows that $p\leq \theta_{\mathrm{start}}\leq 1$.
\begin{proposition}\label{proposition: 0-stop}
	For $a>1$, $p<1$  and $0<\theta<1$ fixed, set
	\[
	f_{\rm 0-stop}(a,\theta):=a(1-\theta+\theta\log\theta).
	\]
	(The reason for this choice of notation will become clear later.) Then, if $f_{\rm 0-stop}(a,\theta)<1$, w.h.p.\ at least one initially
	uninfected vertex has degree less than $\theta a\log n$ in $G_{n,r}^d$ and consequently never becomes infected. On the other hand if $f_{\rm 0-stop}(a,\theta)>1$ then w.h.p.\ the minimum degree of $G_{n,r}^d$ is at least $\theta a \log n$.
\end{proposition}
\begin{proof}
	Write $Y$ for the number of vertices which have fewer than $a\theta\log n$ neighbours in $G_{n,r}^d$. Again, using Wald's identity and Lemma~\ref{lemma: Paul} we have
	\begin{align*}
	\E(Y)&=n\exp\left\{a\log n(\theta-1-\theta\log(\theta))+O(\log_+a\theta\log n)\right\}=\exp\left\{\log n(1-f_{\rm 0-stop}(a,\theta)+o(1))\right\}.
	\end{align*}
	Thus if $f_{\rm 0-stop}(a,\theta)>1$ then $\E(Y)=o(1)$ and by Markov's inequality w.h.p.\ the minimum degree of $G_{n,r}^d$ is at least $\theta a \log n$.

	On the other hand if $f_{\rm 0-stop}(a,\theta)<1$, then let $E$ be the event that the ball of radius $r+\sqrt{d}$ around the origin contains strictly fewer than  $\theta a \log n$ vertices of $\mathcal{P}$ and that the ball of radius $\sqrt{d}$ around the origin contains at least one initially uninfected vertex of $\mathcal{P}$. Then by Lemma~\ref{lemma: Paul} and standard properties of the Poisson point process, the probability of $E$ is
		\begin{align*}
	\mathbb{P}(E)&=\exp\left\{-f_{\rm 0-stop}(a,p,\theta)\log n+o(\log n)\right\} \left(1-\exp\left\{-(1-p)\alpha_d (\sqrt{d})^d \right\}\right)\\
	&= \exp\left\{-f_{\rm 0-stop}(a,p,\theta)\log n+o(\log n)\right\}.
	\end{align*}
By Lemma~\ref{lemma: first and second moment}(ii) w.h.p.\ $N_E=n^{1-f_{\rm 0-stop}+o(1)}$. In particular w.h.p.\ the event $E(\mathbf{x})$ occurs for some $\mathbf{x}\in \Gamma$. This implies there exists an initially uninfected vertex $v\in B_{\sqrt{d}}(\mathbf{x})\cap \mathcal{P}$ such that $B_r(v)\subseteq B_{r+\sqrt{d}}(\mathbf{x})$ contains strictly fewer than $\theta a \log n$ points of $\mathcal{P}$. Thus $v$ never becomes infected and  w.h.p.\ $A_{\infty}\neq \mathcal{P}$. This concludes the proof of the proposition.
\end{proof}

A few comments are in order. First, if in Proposition~\ref{proposition: 0-stop} we have $f_{\rm 0-stop}(a,\theta)=1$, then we cannot apply Lemma~\ref{lemma: Paul}
directly, since the error term will dominate. However, in this case, $\Prb(Y>0)$ will tend to some constant that is neither
0 nor 1. An equivalent remark applies to Proposition~\ref{prop: start}. More precise results can be established in these special cases using the Stein-Chen
method for Poisson approximation; however we will not pursue such questions here. Second, when, say, $f_{\rm start}(a,p,\theta)<1$, not only does the infection
start to spread somewhere, it in fact starts to spread in $n^{1-f_{\rm start}(a,p,\theta)+o(1)}=n^{\alpha+o(1)}$ different places for some $\alpha >0$, as established in the proof of Proposition~\ref{prop: start} (more specifically the lower bound on $N_E$). Third, in Proposition~\ref{proposition: 0-stop},
we have found one obstruction to full infection, but there may (and in fact will) be others.

\section{Bootstrap percolation on the torus}	\label{section: 2d}
In this section we prove our result for the Bradonji\'c--Saniee model for bootstrap percolation in the Gilbert random geometric graph $G_{n,r}^2$ on the $2$-dimensional torus.	Throughout the section we let $T_n:=T^2_n$ denote the said $2$-dimensional torus.

\subsection{Almost percolation from a large local infection: the case $\theta<\frac{1+p}{2}$}
Fix $a>1$, and let $\pi {r}^2 =a \log n$. In this subsection, our goal is to show that if $p, \theta$ are fixed and satisfy the following \emph{growing condition},
\begin{align}\label{eq: growing condition}
\pi \theta< \frac{\pi}{2}(1+p),
\end{align}
then w.h.p.\  any sufficiently large local infection spreads to almost all of $T_n$. To state our formal result (Theorem~\ref{theorem: giant spreads}), we must introduce two tilings of $[0, \sqrt{n}]^2$. Let $K\in \mathbb{N}$ be a large constant to be specified later.

\begin{definition}[Rough tiling, fine tiling]
	The \emph{rough tiling} $\mathcal{R}$ partitions $[0,\sqrt{n}]^2$ into interior-disjoint $Kcr\times Kcr$ square tiles, where $c=\frac{\sqrt{n}}{Kr\lfloor\sqrt{n}/Kr\rfloor}= 1+o(1)$ is chosen to ensure divisibility conditions are satisfied. The \emph{fine tiling} $\mathcal{F}$ is a refinement of $\mathcal{R}$ obtained by subdividing each tile of $\mathcal{R}$ into $K^4$ smaller $\frac{cr}{K}\times \frac{cr}{K}$ square tiles.
\end{definition}

Let $p,\theta $ be fixed. Suppose $p,\theta$ satisfy~\eqref{eq: growing condition}. Then for any $\eta>0$ there exist constants $C_{\eta}>1$ sufficiently large and $\delta>0$ sufficiently small such that the following hold: let $\mathbf{0}$ denote the origin in $\mathbb{R}^2$, and let $R\geq C_{\eta}r$ be a real number. Then for any point $\mathbf{x}$ at distance between $R-\delta r$ and $R+ \delta r$ of $\mathbf{0}$, the asymmetric lens  $B_{R}(\mathbf{0})\cap B_{r}(\mathbf{x})$ has area at least
\begin{align}\label{eq: lens is almost a half disc}
\vert B_{R}(\mathbf{0})\cap B_{r}(\mathbf{x})\vert \geq \frac{\pi r^2}{2}(1-\eta).
\end{align}
We can now specify our choice of $K$. Since $a, p, \theta$ are fixed and satisfy~\eqref{eq: growing condition}, there exists a constant $\eta>0$ such that
\begin{align}\label{eq: growing condition approximated}
\pi \theta< \frac{\pi}{2}(1+p)(1-2\eta)  -\eta.
\end{align}
Fix $\eta>0$ such that~\eqref{eq: growing condition approximated} is satisfied. Let $C_{\eta}$ and $\delta$ be such that~\eqref{eq: lens is almost a half disc} is satisfied. Now set $K=\lceil \max\left(4C_{\eta}, 1000/\eta, 1000/\delta\right)\rceil$.

With $K$ fixed (and with it our rough and fine tilings), we can now define tile colourings which we will use as discrete proxies for the spread of an infection in $T_n$.  We assign colours to the tiles of $\mathcal{R}$ and $\mathcal{F}$ as follows: a tile $T\in \mathcal{F}$ is coloured \textbf{white} if either it contains fewer than $(1-\eta)p\vert T\vert$ initially infected points at the start of the bootstrap percolation process,  or it contains fewer than  $(1-\eta)\vert T\vert $ points in total. Otherwise, we colour $T$ \textbf{red} if all its points are infected by the end of the bootstrap percolation process, and \textbf{blue} if this is not the case. Further, we colour a tile in $\mathcal{R}$ \textbf{white} if one of its subtiles in $\mathcal{F}$ is coloured white, \textbf{red} if all its subtiles in $\mathcal{F}$ are coloured red, and \textbf{blue} otherwise.

We will be interested in the interface between red and non-red tiles in $\mathcal{R}$. We thus equip $\mathcal{R}$ with the natural square-grid graph structure by decreeing that two tiles in $\mathcal{R}$ are adjacent if they meet in a side\footnote{Here we identify $[0, \sqrt{n}]^2$ with $\left(\mathbb{R}/\sqrt{n}\mathbb{Z}\right)^2$ in the natural way to ensure that, as is the case in the torus, the tiles in the rightmost column in $\mathcal{R}$ are adjacent with the corresponding tiles in the leftmost column, and similarly the tiles in the topmost row in $\mathcal{R}$ are adjacent with the corresponding tiles in the bottommost column.}. By a \emph{red component} in $\mathcal{R}$, we mean a connected component of red tiles in the graph $H$ thus defined on $\mathcal{R}$.

We are now in a position to state the main result of this subsection.
\begin{theorem}\label{theorem: giant spreads}
	Let $p,\theta $ be fixed. Suppose $p,\theta$ satisfy~\eqref{eq: growing condition} and let $K$ be as defined above. Then there exist constants $N\in \mathbb{N}$ and $\varepsilon>0$ such that w.h.p.\ if there is a connected component of at least $N$ red tiles in the auxiliary graph $H$ on $\mathcal{R}$, then there exists a giant connected component of $\vert \mathcal{R}\vert -o(n^{1-\varepsilon})$ red tiles in $H$. Furthermore, the non-red tiles consist of a collection of $o(n^{1-\varepsilon})$ vertex-disjoint connected subgraphs of $H$, each of which has order at most $N$.
\end{theorem}
\noindent In other words, once the growing condition is satisfied, any sufficiently large infection spreads to most of  $T_n$, leaving  only $o(n^{1-\varepsilon})$ isolated islands of diameter $O(\log n)$ uninfected.

The proof of Theorem~\ref{theorem: giant spreads} relies on two main ingredients. First of all, we shall use~\eqref{eq: lens is almost a half disc} and some tiling approximations to show that, in the absence of fine white tiles, an infection will spread radially outwards from a sufficiently large infected disc (this is the content of Lemma~\ref{lemma: balls grow}). Next, we shall use the Bollob\'as--Leader discrete isoperimetric inequality in the toroidal grid to show that any large component of rough red tiles has a large boundary. Combining these results with some probabilistic estimates showing that white tiles are few and far apart will then yield the final result.

In the first part of the proof, we shall use the following technical lemma, which follows from~\cite[Lemma 8]{FalgasRavryWalters12}.
\begin{proposition}\label{prop: bound on curve-meeting tiles}
	Let $\Gamma$ be a continuous, piecewise continuously differentiable curve in $T_n$. Let $\ell(\Gamma)$ be the length of the curve $\Gamma$. Then $\Gamma$ meets at most $9K\ell(\Gamma)/r$ tiles of $\mathcal{F}$.	
\end{proposition}
\begin{lemma}[Growing lemma]\label{lemma: balls grow}
	Suppose $\mathbf{x}$ is a point in $T_n$ and $R\geq C_{\eta}r$ is a real number such that the following hold:
	\begin{enumerate}[(i)]
		\item all fine tiles that are wholly contained inside $B_R(\mathbf{x})\setminus B_{R-2r}(\mathbf{x})$ are coloured red;
		\item no fine tile wholly contained inside $B_{R+2r}(\mathbf{x})\setminus B_{R-2r}(\mathbf{x})$ is coloured white.
	\end{enumerate}
	Then all fine tiles that are wholly contained inside $B_{R+\delta r}(\mathbf{x})\setminus B_{R-2r}(\mathbf{x})$ are coloured red.
\end{lemma}
\begin{proof}
	Any fine tile wholly contained inside $B_{R+\delta r}(\mathbf{x})\setminus B_{R-2r}(\mathbf{x})$ is either wholly contained inside $B_R(\mathbf{x})\setminus B_{R-2r}(\mathbf{x})$ (and hence coloured red), or, by our choice of $K$, is wholly contained inside 	$B_{R+\delta r}(\mathbf{x})\setminus B_{R-\delta r}(\mathbf{x})$ (and hence not coloured white).

	It is thus enough to show that every vertex $\mathbf{y} \in B_{R+\delta r}(\mathbf{x})\setminus B_{R-\delta r}(\mathbf{x})\cap \mathcal{P}$ is eventually infected by our bootstrap percolation process. Now by~\eqref{eq: lens is almost a half disc}, the lens $L({\mathbf{y}}):=B_{R}(\mathbf{x})\cap B_r(\mathbf{y})$ has area at least $\frac{\pi r^2}{2}(1-\eta)$. Since this lens is contained inside the annulus  $B_R(\mathbf{x})\setminus B_{R-2r}(\mathbf{x})$, every fine tile wholly contained inside $L({\mathbf{y}})$ is coloured red by Assumption (i). Further, every other tile wholly contained inside the disc $B_r(\mathbf{y})$ is not coloured white by Assumption (ii). We use this information, together with Proposition~\ref{prop: bound on curve-meeting tiles}, to show that $\mathbf{y}$ sees strictly more than $\theta \pi r^2$ infected points within distance at most $r$ of itself --- which in turn implies that $\mathbf{y}$ must become infected before the end of the process, as required.

	First of all,  the boundary of $L({\mathbf{y}})$ has length at most $2\pi r$, whence by Proposition~\ref{prop: bound on curve-meeting tiles} it intersects at most $18 \pi K$ distinct fine tiles. It follows that  $L({\mathbf{y}})$ must wholly contain a collection of red fine tiles of combined area at least
	\begin{align*}
	\vert L({\mathbf{y}})\vert - 18 \pi K\frac{c^2r^2}{K^2}>\frac{\pi r^2}{2}(1-\eta) -\frac{\eta r^2}{2},
	\end{align*}
	where the inequality follows for $n$ large enough from our lower bound on $\vert L({\mathbf{y}})\vert$, our choice of $K\geq 1000/\eta $  and the fact $c=1+o(1)$.

	Similarly, the boundary of $B_{r}(\mathbf{y})$ has length $2\pi r$ and thus, applying Proposition~\ref{prop: bound on curve-meeting tiles} as above, $B_{r}(\mathbf{y})$ must wholly contain a collection of non-white fine tiles of combined area at least $\pi r^2 -\frac{\eta r^2}{2}$. Given the definition of our colouring of fine tiles, it follows that $B_r(\mathbf{y})$ contains at least
	\begin{align*}
	\left(\frac{\pi r^2}{2}(1-\eta) -\frac{\eta r^2}{2}\right)(1-\eta) +\left(\frac{\pi r^2}{2}(1+\eta)\right)p(1-\eta)>\frac{\pi r^2}{2}(1+p)(1-2\eta) -\frac{\eta r^2}{2},
	\end{align*}
	infected points, which by~\eqref{eq: growing condition approximated} is strictly more than $\theta \pi r^2$. Thus $\mathbf{y}$ is eventually infected by the bootstrap percolation process, and the lemma follows. 	
\end{proof}
Lemma~\ref{lemma: balls grow} implies that if a red rough tile is part of the vertex-boundary of a connected component of red rough tiles in the graph $H$ on $\mathcal{R}$, then there must be a white rough tile in the vicinity.
\begin{corollary}\label{corollary: red rough tile has red neighbour or a distance 3 white neighbour}
	Let $T\in \mathcal{R}$ be coloured red. Then either all neighbours of $T$ in the auxiliary graph $H$ on $\mathcal{R}$ are coloured red, or there exists a rough tile $T'$ at graph distance at most $3$ of $T$ in $H$ such that $T'$ is coloured white.
\end{corollary}
\begin{proof}
	Suppose that no rough tile $T'$ at graph distance at most $3$ of $T$ in $H$ is coloured white. Let $\mathbf{x}$ denote the centre of the tile $T$. By assumption, every fine tile wholly contained inside the ball of radius $\frac{Kcr}{2}>C_{\eta}r$ about $\mathbf{x}$ is coloured red. Further, no fine tile wholly contained inside the disc of radius $\sqrt{10}Kcr/2 +2r$ about $\mathbf{x}$ is coloured white. Since the disc $B_{\sqrt{10}Kcr/2}(\mathbf{x})$ wholly contains the four rough tiles adjacent to $T$ in the auxiliary graph $H$, it follows from $\lceil (\sqrt{10}/2-1/2)Kc/\delta\rceil $ successive applications of Lemma~\ref{lemma: balls grow} that all fine tiles lying inside the four neighbours of $T$ in $H$ are coloured red, and hence that all the neighbours of $T$ in $H$ are coloured red themselves.
\end{proof}
We shall use Corollary~\ref{corollary: red rough tile has red neighbour or a distance 3 white neighbour} to show that if a connected component of red tiles in $H$ has a large boundary, then we may find a large connected component of white tiles in a sufficiently large power of $H$. Recall that the $t$-th power of $H$, denoted by $H^t$, is the graph on the vertex set of $H$ in which all pairs of distinct vertices lying at graph distance at most $t$ in $H$ are joined by an edge.

To make this argument formal, we need the standard notions of edge boundary and dual cycles. Given a subset $A\in \mathcal{R}$, the \emph{edge boundary} $\partial_e(A)$ is the collection of pairs $(T_1, T_2)$ such that $T_1\in A$, $T_2\notin A$ and $\{T_1,T_2\}$ is an edge of $H$. The dual $H^{\star}$ of the graph $H$ is the graph whose vertices are the corners of rough tiles in $\mathcal{R}$ and whose edges are the sides of rough tiles in $\mathcal{R}$. It can be shown (see e.g.~\cite[Lemma 1, Chapter 1]{BollobasRiordan06}) that the edge boundary of a connected component $\mathcal{C}$ in $H$ corresponds to a union of cycles in the dual graph $H^{\star}$.

We can now prove that to each cycle $\mathcal{C}^{\star}$ in the edge boundary of a red component $C$ in $H$, we may associate a (large) connected component of white tiles in the seventh power $H^7$ of $H$.
\begin{lemma}\label{lemma: large boundary implies large white comp in H7}
	Let $\mathcal{C}$	be a connected component of red tiles in $H$. Let $\mathcal{C}^{\star}$ be a dual cycle of length $\ell$ in the edge boundary of $\mathcal{C}$ in $H$. Let $T_0$ be an arbitrary red tile in $\mathcal{C}$ one of whose sides corresponds to an edge of the dual cycle $\mathcal{C}^{\star}$. Then there exists a connected component of white rough tiles in $H^{7}$ of order at least $\min(1,\ell/100)$, one of whose tiles is at graph distance at most $3$ of $T_0$ in $H$.
\end{lemma}
\begin{proof}
	By Corollary~\ref{corollary: red rough tile has red neighbour or a distance 3 white neighbour}, we know that for every pair $e_i=(T^i_1, T^i_2)$ in the edge boundary of $\mathcal{C}$, there exists a white rough tile $T_i$ within graph distance at most $3$ of $T^i_1$. Let $e^{\star}_1, e^{\star}_2, \ldots e^{\star}_{\ell}$ be the edges of the dual cycle $\mathcal{C}^{\star}$, and let $e_1, e_2, \ldots, e_{\ell}$ be the corresponding pairs from $\partial_e(\mathcal{C})$. Assume without loss of generality that $e^{\star}_1$ is a side of $T_0$.

	Traversing the edges $e^{\star}_1, e^{\star}_2, \ldots, e^{\star}_{\ell}$ of $\mathcal{C}^{\star}$ in order, we may thus obtain a sequence of rough white tiles $T_1, T_2, \ldots, T_{\ell}$, where $T_i$ and $T_{i+1}$ are at graph distance at most $7$ of each other in $H$. Since there are $25$ rough tiles within distance at most $3$ of a given white rough tile in $H$, it follows that no white tile may be repeated more than $100$ times in our sequence (since each tile has four sides, each of which occurs at most once as an edge $e^{\star}_i$ in $\mathcal{C}^{\star}$). Thus there must be a component of at least $\ell /100$ rough white tiles in $H^7$, one of which (namely $T_1$) is within graph distance at most $3$ of $T_0$ in $H$.
\end{proof}
Thus if a red component in $H$ has a large dual cycle in its boundary, there must exist a large white component in $H^7$. On the other hand, as we now show,  w.h.p.\  there are no large white components in $H^7$. Set $k:=\lfloor \sqrt{n}/Kr\rfloor$. Our auxiliary graph $H$ on the set of rough tiles $\mathcal{R}$ is thus a $k\times k$ toroidal graph.
\begin{lemma}\label{lemma: white components small}
	There exists a constant $\varepsilon=\varepsilon (a, \eta, K)>0$ such that w.h.p.\ the following hold:
	\begin{enumerate}[(i)]
		\item there are $O(n^{1-\varepsilon}/\log n)$ white rough tiles in $H$;
		\item connected components of white rough tiles in $H^7$ have diameter at most $1/\varepsilon$ in $H^7$.
	\end{enumerate}
\end{lemma}
\begin{proof}
	By Lemma~\ref{lemma: Paul} (applied to the $K^4$ fine subtiles of a rough tile) and Markov's inequality,	there exists a constant $\varepsilon=\varepsilon (a, \eta, K)>0$ such that the probability that a rough tile is coloured white is at most $n^{-\varepsilon}$ for all $n$ sufficiently large. Thus the expected number of white rough tiles in $H$ is at most $k^2n^{-\varepsilon}=O\left(n^{1-\varepsilon}/\log n\right)$. Applying Markov's inequality again, we obtain the first part of the lemma.

	For the second part, observe that the graph $H^7$ has maximum degree less than $4^7$, and that each tile is coloured white independently of all other tiles. In particular, the expected number of paths of white rough tiles of length $\ell>1/\varepsilon$ in $H^7$ is at most
	\begin{align*}
	k^2 4^{7\ell} n^{-(\ell +1)\varepsilon}=O(n^{-\varepsilon})=o(1),
	\end{align*}
	whence Markov's inequality tells us that w.h.p.\ no such path exists. It follows that w.h.p.\ all connected components of white rough tiles in $H^7$ have diameter at most $1/\varepsilon$ in $H^7$, as claimed.
\end{proof}
Putting Lemmas~\ref{lemma: large boundary implies large white comp in H7} and~\ref{lemma: white components small} together, we have that w.h.p.\ there is no red component in $H$ with a large dual cycle in its edge boundary. We now bring in the second main ingredient of the proof of Theorem~\ref{theorem: giant spreads}, namely a discrete isoperimetric inequality in the toroidal grid due to Bollob\'as and Leader~\cite{BollobasLeader91}, in order to show that in such circumstances if there is a large red component $\mathcal{C}$ in $H$, then this component $\mathcal{C}$ is unique and all components in $H-\mathcal{C}$ are small.

The $2$-dimensional case of the Bollob\'as--Leader edge-isoperimetric inequality for toroidal grids~\cite[Theorem 8]{BollobasLeader91} is as follows.
\begin{proposition}[Bollob\'as--Leader edge-isoperimetric inequality]\label{proposition: bollobas leader}
	Let $A$ be a subset of $\mathcal{R}$ with $\vert A\vert  \leq k^2/2$. Then  	
	\begin{align*}
	\vert \partial_e(A)\vert &\geq  \min \left(4\sqrt{\vert A\vert }, 2k\right).
	\end{align*}	
\end{proposition}
\begin{lemma}\label{lemma: if large comp and no large boundary cycles, then all comps in complement small}
	Let $N_0\in \mathbb{N}$ be fixed. Then for every $n$ sufficiently large, the following holds: if $\mathcal{C}$ is a connected component of order at least $(N_0)^2$ in $H$ and has the property that every dual cycle in the edge-boundary of $\mathcal{C}$ has length at most $N_0$,  then every connected component in $H[V(H)\setminus \mathcal{C}]$ has order strictly less than $(N_0)^2$.
\end{lemma}
\begin{proof}
	Let $N_0$ be fixed, and let $n$ be sufficiently large so as to ensure $2k>N_0$. Let $\{\mathcal{C}_i: \ i \in I\}$ be the collection of connected components in $H[V(H)\setminus \mathcal{C}]$. Note that each such component $\mathcal{C}_i$ sends an edge to $\mathcal{C}$ in $H$, and sends no edge to $\mathcal{C}_j$ for $j\neq i$. For every $i_0\in I$, consider the edge-boundary of $\mathcal{C}_{i_0}$. Since both $\mathcal{C}_{i_0}$ and its complement $\mathcal{C}\cup \bigcup_{i\in I\setminus\{i_0\}}\mathcal{C}_i$ induce connected subgraphs in $H$, the edge-boundary of $\mathcal{C}_{i_0}$ (which is a subset of the edge-boundary of $\mathcal{C}$) consists of a single dual cycle $\mathcal{C}^{\star}_{i_0}$ in $H^{\star}$.

	If $\vert \mathcal{C}_{i_0}\vert > k^2/2$, then by
Proposition~\ref{proposition: bollobas leader} applied to the complement $\mathcal{C}\cup \bigcup_{i\in I\setminus\{i_0\}}\mathcal{C}_i$ of $\mathcal{C}_{i_0}$, we have that this dual cycle $\mathcal{C}^{\star}_{i_0}$ has length at least
	\[\vert \partial_e (\mathcal{C}_{i_0})\vert \geq \min \left(4\sqrt{k^2-\vert \mathcal{C}_{i_0}\vert} , 2k\right)\geq \min\left(4\sqrt{\vert \mathcal{C}\vert}, 2k\right)>N_0,\]
	by our assumptions that $\vert \mathcal{C}\vert \geq (N_0)^2$ and $2k>N_0$. This contradicts the fact that every dual cycle in the edge boundary of $\mathcal{C}$ has length at most $N_0$. Thus it must be the case that $\vert \mathcal{C}_{i_0}\vert \leq k^2/2$.

	Then applying
Proposition~\ref{proposition: bollobas leader} to $\mathcal{C}_{i_0}$ itself, we have that the dual cycle $\mathcal{C}^{\star}_{i_0}$ has length at least
	\[\vert \partial_e (\mathcal{C}_{i_0})\vert \geq \min \left(4\sqrt{\vert \mathcal{C}_{i_0}\vert} , 2k\right).\]
	Since by assumption every dual cycle in the edge boundary of $\mathcal{C}$ has length at most $N_0$, and since $2k> N_0$, it follows from the above that $\vert \mathcal{C}_{i_0}\vert \leq (N_0)^2/16<(N_0)^2$, as required.
\end{proof}
\noindent We are now ready to complete the proof of Theorem~\ref{theorem: giant spreads}.
\begin{proof}[Proof of Theorem~\ref{theorem: giant spreads}]
	Let $\varepsilon>0$ be the constant whose existence is guaranteed by Lemma~\ref{lemma: white components small}. Now, let $N_0= 100 \left(4^7\right)^{1/\varepsilon}$ and $N=(N_0)^2$. By Lemma~\ref{lemma: white components small} and the fact that $H^7$ has maximum degree less than $4^7$, we have that w.h.p.\  (A1) there are $O(n^{1-\varepsilon}/\log n)$ white rough tiles in $H$, and all connected components of white rough tiles in $H^7$ have order at most $N_0/100$. By Lemma~\ref{lemma: large boundary implies large white comp in H7}, this in turn implies that w.h.p.\  (A2) all dual cycles in the edge boundary of red components have length at most $N_0$.

	Assume (A1) and (A2) both hold, and suppose $H$ contains a red component $\mathcal{C}$ of order $\vert \mathcal{C}\vert \geq N$. Then by Lemma~\ref{lemma: if large comp and no large boundary cycles, then all comps in complement small}, all connected components in $H-\mathcal{C}:=H[V(H)\setminus \mathcal{C}]$ have order strictly smaller than $N=(N_0)^2$. In particular, $\mathcal{C}$ is unique: all other red components must have order strictly less than $N$.

	Furthermore, we can bound the number of components in $H-\mathcal{C}$: by Lemma~\ref{lemma: large boundary implies large white comp in H7}, to each connected component $\mathcal{C}'$ whose edge boundary with $\mathcal{C}$ contains a pair $(T_0, T)$ with $T_0\in \mathcal{C}$ and $T\in \mathcal{C}'$ we may associate a connected component of white rough tiles $W$ in $H^7$, one of whose tiles is at graph distance at most $3$ of $T_0$ in $H$.

	By (A1), there are at most $O(n^{1-\varepsilon}/\log n)$ connected components of white rough tiles $W$ in $H_7$. Further, for each such $W$, there are at most $\vert W\vert 4^3\leq \frac{64}{100}N_0< N_0$ red tiles $T_0$ within graph distance at most $3$ in $H$ of a white tile in $W$. It follows from this and the remarks in the paragraph above that there are at most $4N_0 \times O(n^{1-\varepsilon}/\log n)=o(n^{1-\varepsilon})$ connected components $\mathcal{C}'$ in $H-\mathcal{C}$, each of which has order at most $N$ (thereby establishing the `furthermore' part of Theorem~\ref{theorem: giant spreads}). It follows that
	\begin{align*}
	\vert C\vert \geq \vert \mathcal{R}\vert-o(n^{1-\varepsilon}),
	\end{align*}
	i.e.\ that $\mathcal{C}$ is a giant red connected component covering all but $o(n^{1-\varepsilon})$ tiles in $H$, as claimed. This concludes the proof of Theorem~\ref{theorem: giant spreads}.
\end{proof}
\subsection{Outbreaks stay local: the case $\theta>\frac{1+p}{2}$}
Fix $a>1$, and let $\pi {r}^2 =a \log n$. In this subsection, our goal is to show that if $p, \theta$ are fixed and satisfy the following \emph{non-growing condition},
\begin{align}\label{eq: non-growing condition}
\pi \theta> \frac{\pi}{2}(1+p),
\end{align}
then w.h.p. any infectious outbreak in $T_n:=T_n^2$ remains local. To state our formal result (Theorem~\ref{theorem: no giant}), we must, as in the previous subsection, introduce two tilings of $[0,\sqrt{n}]^2$. Let $K\in \mathbb{N}$ be a large constant to be specified later.
\begin{definition}[Rough tiling, fine tiling]
	The rough tiling $\mathcal{R}$ partitions $[0,\sqrt{n}^2]$ into disjoint $cKr\times cKr$ square tiles, where $c=\frac{\sqrt{n}}{Kr\lfloor \sqrt{n}/Kr\rfloor}=1+o(1)$ is chosen to ensure divisibility conditions are satisfied. The \emph{fine tiling} $\mathcal{F}$ is a refinement of $\mathcal{R}$ obtained by subdividing each tile of $\mathcal{R}$ into $K^4$ smaller $\frac{cr}{K}\times \frac{cr}{K}$ square tiles.
\end{definition}
\begin{remark}
	Note that these tilings are technically distinct from those we used in the previous subsection. (We will pick a different value of $K$.)
\end{remark}	
Let $p,\theta $ be fixed. Suppose $p,\theta$ satisfy~\eqref{eq: non-growing condition}. Then for any $\eta>0$ there exists a constant $C_{\eta}>1/\sqrt{2}$ sufficiently large such that the area of the lune $B_{r}((C_{\eta}r,0))\setminus B_{C_{\eta}r}(\mathbf{0})$ is at most
\begin{align}\label{eq: lune not much larger than a half disc}
\vert B_{r}((C_{\eta}r,0))\setminus B_{C_{\eta}r}(\mathbf{0})\vert \leq \frac{\pi r^2}{2}(1+\eta)
\end{align}
\noindent We can now specify our choice of $K$. Since $a,p, \theta$ are fixed and satisfy~\eqref{eq: non-growing condition}, there exists a constant $\eta>0$ such that
\begin{align}\label{eq: non-growing condition approximated}
\pi \theta> \frac{\pi}{2}(1+p)(1+2\eta)  +\eta.
\end{align}
Fix $\eta>0$ such that~\eqref{eq: non-growing condition approximated} is satisfied. Let $C_{\eta}$ be such that~\eqref{eq: lune not much larger than a half disc} is satisfied. Now set $K=\lceil \max\left(2C_{\eta}, 10000/\eta\right)\rceil$.

With $K$ fixed (and with it our rough and fine tilings), we can now define tile colourings which we will use as discrete proxies for the spread of an infection in $T_n$.  We assign colours to the tiles of $\mathcal{R}$ and $\mathcal{F}$ as follows: a tile $T\in \mathcal{F}$ is coloured \textbf{black} if either it contains strictly more than $(1+\eta)p\vert T\vert$ points of $A_0$, or it contains strictly more than  $(1+\eta)\vert T\vert $ points of $\mathcal{P}$. Otherwise, we colour $T$ \textbf{red} if some of its initially uninfected points become infected at some stage in the bootstrap percolation process, and \textbf{blue} if this is not the case. Further, we colour a tile in $\mathcal{R}$ \textbf{black} if one of its subtiles in $\mathcal{F}$ is coloured black, \textbf{red} if one of its subtiles in $\mathcal{F}$ is coloured red, and \textbf{blue} otherwise.

We equip $\mathcal{R}$ with the natural square-grid graph structure by decreeing that two tiles in $\mathcal{R}$ are adjacent if they meet in a side, and as in the previous subsection, we identify $[0,\sqrt{n}]^2$ with $\left(\mathbb{R}/\sqrt{n}\mathbb{Z}\right)^2$ in the natural way. We thus obtain an auxiliary toroidal grid-graph $H$ on $\mathcal{R}$. We can now state the main result of this subsection.
\begin{theorem}\label{theorem: no giant}
	Let $p, \theta$ be fixed. Suppose $p,\theta$ satisfy~\eqref{eq: non-growing condition}, and let $K$ and our rough tiling be as defined above. Then there exists $\varepsilon>0$ such that for any constant $N$, even if one fully infects all points inside $N$ adversarially chosen tiles of $\mathcal{R}$, all but $o(n^{1-\varepsilon})$ tiles of $\mathcal{R}$ are coloured blue. Furthermore, the non-blue tiles consist of a collection of $o(n^{1-\varepsilon})$ vertex-disjoint subgraphs of $H$, each of which has order at most $100N^2/\varepsilon^2$.
\end{theorem}

The key to Theorem~\ref{theorem: no giant} is the following lemma, showing that a non-black rough tile surrounded by non-black tiles will be coloured blue.
\begin{lemma}\label{lemma: square non-black implies center blue}
Suppose $T$ is a tile in $\mathcal{R}$. Suppose none of the tiles in the $3\times 3$ square grid of tiles of $\mathcal{R}$ centred at $T$ is coloured black. Then $T$ is coloured blue.
\end{lemma}
\begin{proof}
Let $\mathbf{x}$ denote the centre of the tile $T$. We shall prove the stronger claim that no initially uninfected vertex in the ball $B_{cKr/\sqrt{2}}(\mathbf{x})\supseteq T$ ever becomes infected.

Indeed, suppose $t\geq 0$ and no point in $B_{cKr/\sqrt{2}}(\mathbf{x})\setminus A_0$ has yet become infected. Consider any such point $v$. Which infected points does $v$ see within distance $r$ of itself? In a worst-case scenario, every point in the lune $\mathrm{Lune}(v)=B_r(v)\setminus B_{cKr/\sqrt{2}}(\mathbf{x})$ has become infected. We show that even if this was the case, $v$ does not become infected in the next round of the bootstrap percolation process.

Observe first of all that the length of the boundary $\partial\mathrm{Lune}(v)$ of the lune $\mathrm{Lune}(v)$ is at most $2\pi r$, and thus $\partial\mathrm{Lune}(v)$ meets at most $18\pi K$ tiles of $\mathcal{F}$ by Proposition~\ref{prop: bound on curve-meeting tiles}. Similarly, the boundary of the asymmetric lens $\mathrm{Lens}(v)=B_r(v)\cap B_{cKr/\sqrt{2}}(\mathbf{x})$ has length at most $2\pi r$ and thus meets at most $18\pi K$ tiles of $\mathcal{F}$.

Since both $T$ and the eight tiles around it are not coloured black, each fine tile wholly contained in $\mathrm{Lens}(v)$ contains at most $(1+\eta)p\frac{c^2r^2}{K^2}$ points of $A_0$, and contains no other point of $A_t$ by our assumption. Further, every other fine tile having non-empty intersection with $B_r(v)$ contains at most $(1+\eta)\frac{c^2r^2}{K^2}$ points of $\mathcal{P}$ in total, and thus at most that many points of $A_t$.

By~\eqref{eq: lune not much larger than a half disc} and our choices of $\eta$ and $K$, we have that the area of $\mathrm{Lune}(v)$ is at most $\frac{\pi r^2}{2}(1+\eta)$ (since the area of the lune is maximised if $v$ lies on the circle of radius $cKr/\sqrt{2}$ about $\mathbf{x}$, and $cK/\sqrt{2}>C_{\eta}$). It follows that for $n$ large enough the number of infected points of $A_t$ within distance $r$ of $v$ is at most
\begin{align*}
\vert B_r(v)\cap A_t\vert &\leq (1+\eta)\vert \mathrm{Lune}(v)\vert + (1+\eta)p(\pi r^2 -\vert \mathrm{Lune}(v)\vert) +36\pi K(1+\eta)\frac{c^2r^2}{K^2}\\
&< \pi r^2 \left(\frac{1+p}{2}\right)(1+2\eta) + \eta r^2< \theta \pi r^2.
\end{align*}
Here, the first strict inequality follows from our bound on the area of $\mathrm{Lune}(v)$, the facts that $\eta\leq \pi$  and $c<2$ for $n$ large enough, and from our choice of $K$ ensuring $144 (1+\pi)\pi /K< \eta$. The second strict inequality follows from~\eqref{eq: non-growing condition approximated}.

In particular, $v$ sees strictly fewer than $\theta a \log n$ infected points from $A_t$, and does not become infected in the next round of the bootstrap percolation process. Since $v\in B_{cKr/\sqrt{2}}(\mathbf{x})\setminus A_0$ was arbitrary, it follows by induction on $t$ that $B_{cKr/\sqrt{2}}(\mathbf{x})\setminus A_0=B_{cKr/\sqrt{2}}(\mathbf{x})\setminus A_{\infty}$. Since the rough tile $T$ is a subset of $B_{cKr/\sqrt{2}}(\mathbf{x})$ and is not coloured black, it follows that $T$ is coloured blue as claimed.
 \end{proof}
Similarly to Lemma~\ref{lemma: white components small}, we now prove:
\begin{lemma}\label{lemma: black components small}
Let $p\in (0,1)$ be fixed. Then there exists a constant $\varepsilon=\varepsilon (a, \eta, K)>0$ such that w.h.p.\ the following hold:
\begin{enumerate}[(i)]
	\item there are $O(n^{1-\varepsilon}/\log n)$ black rough tiles in $H$;
	\item connected components of black rough tiles in $H^7$ have diameter at most $1/\varepsilon$.
\end{enumerate}
\end{lemma}
\begin{proof}
		By Lemma~\ref{lemma: Paul} (applied to the $K^4$ fine subtiles of a rough tile) and Markov's inequality,	there exists a constant $\varepsilon=\varepsilon (a, \eta, K)>0$ such that the probability that a rough tile is coloured black is at most $n^{-\varepsilon}$ for all $n$ sufficiently large. Thus the expected number of black rough tiles in $H$ is at most $\left(\frac{\sqrt{n}}{Kr}\right)^2n^{-\varepsilon}=O\left(n^{1-\varepsilon}/\log n\right)$. Applying Markov's inequality again, we obtain the first part of the lemma.

	For the second part, observe that the graph $H^7$ has maximum degree less than $4^7$, and that each rough tile is coloured black independently of all other tiles. In particular, the expected number of paths of black rough tiles of length $\ell>1/\varepsilon$ in $H^7$ is at most
	\begin{align*}
	\left(\frac{\sqrt{n}}{Kr}\right)^2 4^{7\ell} n^{-(\ell +1)\varepsilon}=O(n^{-\varepsilon})=o(1),
	\end{align*}
	whence Markov's inequality tells us that w.h.p.\ no such path exists. It follows that w.h.p.\ all connected components of black rough tiles in $H^7$ have diameter at most $1/\varepsilon$, as claimed.
\end{proof}
We can now prove Theorem~\ref{theorem: no giant}.
\begin{proof}[Proof of Theorem~\ref{theorem: no giant}]
Let $\varepsilon=\varepsilon(\eta, K)>0$ be as in Lemma~\ref{lemma: black components small}. Consider the non-blue tiles in $\mathcal{R}$. By Lemma~\ref{lemma: square non-black implies center blue}, every non-blue tile must either be within distance at most $2$ in $H$ of either a black tile or of one of the at most $N$ adversarially infected tiles.

By Lemma~\ref{lemma: black components small}(i), this immediately implies that w.h.p.\ all but $o(n^{1-\varepsilon})$ tiles of $\mathcal{R}$ are coloured blue. Furthermore, by Lemma~\ref{lemma: black components small}(ii), w.h.p.\ every connected component of black rough tiles has diameter at most $1/\varepsilon$ in $H^7$. It follows that every component of black or adversarially infected rough tiles has diameter at most $(N+1)/\varepsilon$ in $H^7$, and hence order at most $4N^2/\varepsilon^2$.

Now, by Lemma~\ref{lemma: square non-black implies center blue}, to every connected component of non-blue tiles in $H^2$, one may associate a connected component of  black or adversarially infected rough tiles in $H^7$ (since every non-blue tile must be within distance $2$ of a black or adversarially infected rough tile). Since there are fewer than $25$ rough tiles within distance at most $2$ in $H$ of a given rough tile, it follows from our earlier bound on the order of connected components of black or adversarially infected rough tiles in $H^7$ that every connected component of non-blue tiles in $H^2$ must have order at most $100N^2/\varepsilon^2$. This concludes the proof of the theorem.
\end{proof}

\subsection{Proof of Theorem~\ref{theorem: (1+p)/2 threshold}}
With our tiling results Theorems~\ref{theorem: giant spreads} and~\ref{theorem: no giant} in hand, we can prove the main results of this paper.
\begin{proof}[Proof of Theorem~\ref{theorem: (1+p)/2 threshold}(i)]
Suppose  $(a,p, \theta)$ is fixed with $a>1$ and $(p,\theta)$ satisfying the growing condition~\eqref{eq: growing condition}. Then there exists $\eta=\eta(p, \theta)>0$ such that~\eqref{eq: growing condition approximated} is satisfied. We can then define $K=K(\eta)$ and the rough tiling $\mathcal{R}$ of $T_n$ as in Theorem~\ref{theorem: giant spreads}. Let $N\in \mathbb{N}$ and $\varepsilon>0$ be constants such that the conclusions of Theorem~\ref{theorem: giant spreads} hold (note our choice of these constants depends only on $a,p, \theta$). Recall from the proof of Theorem~\ref{theorem: giant spreads} (or more specifically of Lemma~\ref{lemma: white components small}) that we may also ensure with our choice of $\varepsilon$ that the probability that a rough tile in $\mathcal{R}$ is coloured white is at most $n^{-\varepsilon}$, provided $n$ is taken sufficiently large.

Now we select a constant $C=C(a, p, \theta)$ sufficiently large such that any ball of radius $Cr$ in $T_n$ wholly contains a connected component in $H$ of at least $N$ non-white rough tiles of $\mathcal{R}$. This requires a little calculation. Set $M=\lceil\frac{2N}{\varepsilon}\rceil$. Observe that there are at most $n$ ways of choosing an $M  \times M$ square grid of rough tiles in $\mathcal{R}$. Taking a simple union bound, the probability that there are at least $\frac{2}{\varepsilon}N$ white rough tiles in such an $M\times M$ grid is at most
\begin{align*}
2^{M^2} \left(n^{-\varepsilon}\right)^{\frac{2}{\varepsilon}N}=O(n^{-2N}).
\end{align*}
By Markov's inequality, it follows that w.h.p.\ every $M\times M$ grid of rough tiles in $\mathcal{R}$ contains fewer than $\frac{2}{\varepsilon}N$ white rough tiles. In particular, every such grid contains a column of $\frac{2}{\varepsilon}N>N$ non-white rough tiles.

We now take $C$ sufficiently large so as to ensure that every ball of radius $Cr$ wholly contains an $M\times M$ square grid of rough tiles. This is easily done: $C= 4MKc$ will certainly do, for instance. It follows that if we infect any ball $B$ of radius $Cr$ in $T_n$, then, changing the colour of the non-white rough tiles wholly contained inside $B$ to red, this yields w.h.p.\ a connected component of red rough tiles of order at least $N$ in $H$. Applying Theorem~\ref{theorem: giant spreads}, we obtain that all but $o(n^{1-\varepsilon})$ of the rough tiles from $\mathcal{R}$ are coloured red following the bootstrap percolation process started from the initially infected set $A_0\cup (B\cap \mathcal{P})$.

By Lemma~\ref{lemma: Paul},  w.h.p.\ every rough tile in $\mathcal{R}$ contains at most $O(\log n)$ vertices, so we deduce from the above that all but $O(n^{1-\varepsilon}\log n)=o(n)$ vertices of $\mathcal{P}$ eventually become infected.

This only leaves the ``Furthermore'' part of Theorem~\ref{theorem: (1+p)/2 threshold}(i) to establish. Consider a component of non-red rough tiles in $H^2$. Such a component is a union of disjoint components of non-red rough tiles in $H$, connected by paths of length $2$ in $H$ where the middle tile $T_0$ in the path is red and the two other tiles are non-red and belong to distinct components. By Corollary~\ref{corollary: red rough tile has red neighbour or a distance 3 white neighbour}, we know that there must be a white tile at distance at most $3$ from $T_0$. Much as in the proof of Theorem~\ref{theorem: giant spreads}, it then follows that to the edge boundary \emph{in $H$} of an $H^2$-connected component of non-red tiles we may associate an $H^7$-connected component of white tiles $W$.  By Lemma~\ref{lemma: white components small}, we know that w.h.p.\ all such components $W$ have diameter at most $1/\varepsilon$.

There are at most $\vert W\vert 4^3$ red tiles $T_0$ within graph-distance $3$ in $H$ of a white tile in $W$, each red tile $T_0$ is the middle point of a path of length $2$ in $H$ between at most $4$ distinct non-red components in $H$, and by Theorem~\ref{theorem: giant spreads} w.h.p.\ all non-red components in $H$ have order at most $N$. It follows that all connected components of non-red tiles in $H^2$ have order at most $(4^{7})^{1/\varepsilon}4^3 4N$.

Now since $K>1$, to each connected component of forever-uninfected vertices in $G_{n,r}[\mathcal{P}\setminus A_0]$  is associated a collection of non-red rough tiles forming a (subset of) a connected component in $H^2$. It follows that the Euclidean diameter in $T_n$ of such a component is bounded above by $(4^{7})^{1/\varepsilon}4^3 4N \sqrt{2Kcr}=O(\sqrt{\log n})$, as claimed. This concludes the proof of Theorem~\ref{theorem: (1+p)/2 threshold}(i).

\end{proof}

\begin{proof}[Proof of Theorem~\ref{theorem: (1+p)/2 threshold}(ii)]
Suppose  $(a,p, \theta)$ is fixed with $a>1$ and $(p,\theta)$ satisfying the non-growing condition~\eqref{eq: non-growing condition}. Then there exists $\eta=\eta(p, \theta)>0$ such that~\eqref{eq: non-growing condition approximated} is satisfied. We can then define $K=K(\eta)$ and the rough tiling $\mathcal{R}$ of $T_n$ as in Theorem~\ref{theorem: no giant}. Let $\varepsilon>0$ be such that the conclusion of Theorem~\ref{theorem: no giant} holds (note our choice of $\varepsilon$ depends only on $a,p, \theta$).

Let $C>1$ be fixed. Let $B$ be an arbitrarily chosen ball of radius $Cr$ in $T_n$. By Proposition~\ref{prop: bound on curve-meeting tiles}, $B$ meets at most
\begin{align*}
N= \left\lceil\frac{\pi (Cr)^2}{(Kcr)^2} + 9K(2\pi Cr)/r \right\rceil<100KC^2
\end{align*}
tiles of $\mathcal{R}$. Fully infect all points inside these at most $N$ tiles and apply Theorem~\ref{theorem: no giant} to deduce that even if all points in $\mathcal{P}\cap B$ are infected, then w.h.p.\  all but $o(n^{1-\varepsilon})$ tiles of $\mathcal{R}$ are coloured blue.

By Lemma~\ref{lemma: Paul} and a union-bound, w.h.p.\ every tile of $\mathcal{R}$ contains at most $O(\log n)$ points of $\mathcal{P}$. Since the point-set $A_{\infty}\setminus (A_0\cup B)$ is contained inside the union of the non-blue tiles of $\mathcal{R}$, it follows that $A_{\infty}\setminus (A_0\cup B)$ contains at most $o(n^{1-\varepsilon} \log n)=o(n)$ points of $\mathcal{P}$.

This only leaves the last part of Theorem~\ref{theorem: (1+p)/2 threshold}(ii) to establish. Here, note that we established in the proof of Theorem~\ref{theorem: no giant} that in fact all connected components of non-blue tiles in $H^2$ have order at most $100N^2/\varepsilon^2=O(1)$. Since $K>2$, the tiles containing points of a connected component in $G_{n,r}^2[A_{\infty}\setminus (A_0\cup B)]$ must form a connected component of non-blue tiles in $H^2$. It immediately follows that each such connected component has Euclidean diameter $O(\sqrt{\log n})$ in $T_n$. This concludes the proof of Theorem~\ref{theorem: (1+p)/2 threshold}(ii).
\end{proof}



	\subsection{Starting a sufficiently large local outbreak: proof of Theorem~\ref{theorem: threshold for symmetric local growth}}\label{subsection: local growth}
With Theorem~\ref{theorem: (1+p)/2 threshold} in hand, we now turn our attention to the problem of determining when (for which triples $(a, p, \theta)$) a large local outbreak will occur w.h.p.\ in the Bradonji\'c--Saniee model for bootstrap percolation on random geometric graphs. We are unfortunately unable to answer this question rigorously, but we can relate it to the solution of an optimisation problem which we conjecture gives the correct threshold for `large' local outbreaks.

\begin{figure}[htp!]
\centering
\includegraphics[width=0.37\linewidth]{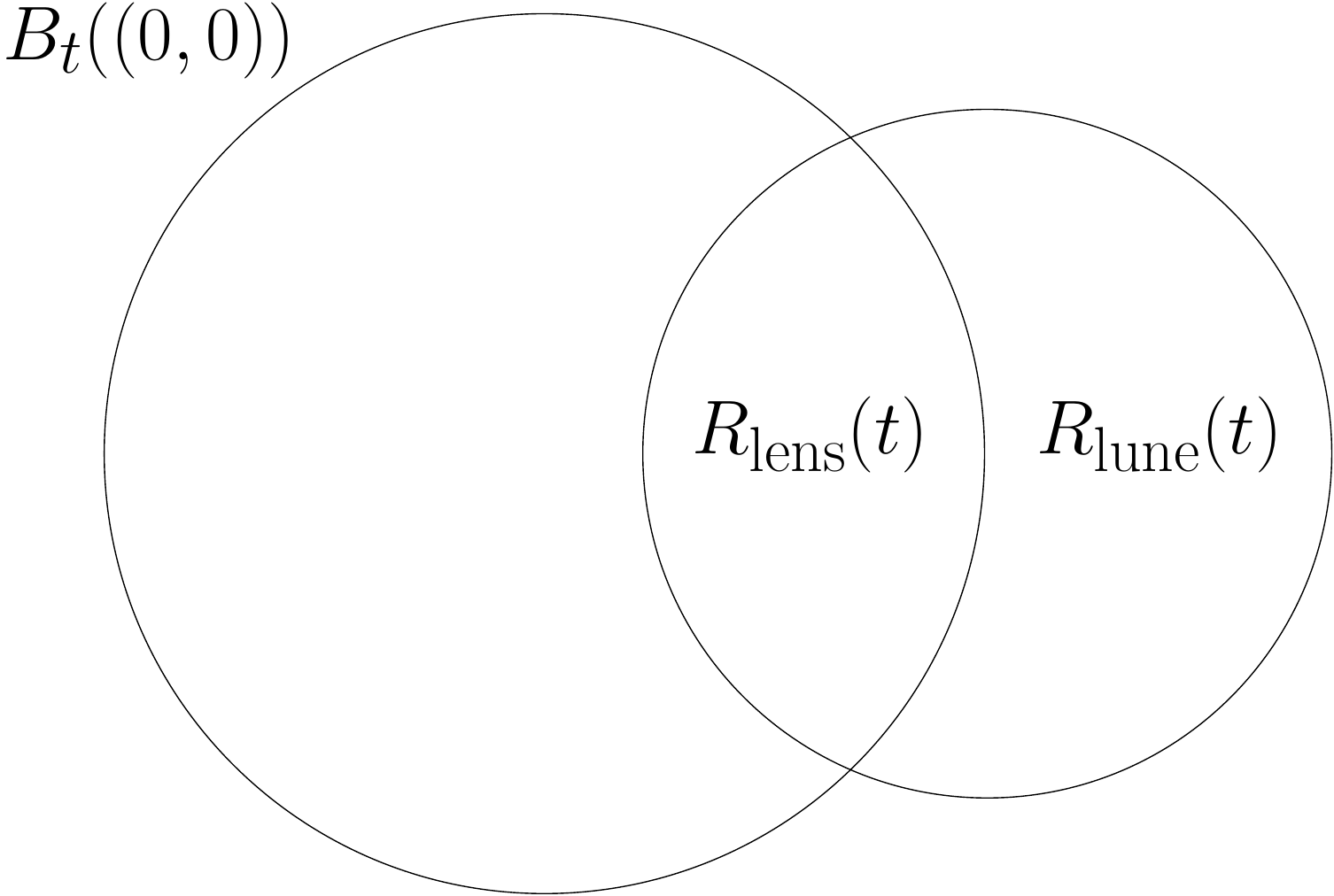}
\caption{\small A lens and a lune}
\label{ll}
\end{figure}

\begin{definition}\label{I and Q definitions}
For $t\in \mathbb{R}_{\geq 0}$, let $R_{\mathrm{lens}}(t)$ denote the asymmetric lens $B_{t}((0,0))\cap B_{1/\sqrt{\pi}}((t,0))$. Let also $R_{\mathrm{lune}}(t)$ denote the lune $B_{1/\sqrt{\pi}} ((t,0))\setminus B_{t}((0,0))$ (see Figure~\ref{ll} for an illustration).
Given functions $f,g: \ \mathbb{R}_{\geq 0}\rightarrow \mathbb{R}_{\geq 1}$, we define the quantities
\[I(f,g)(t):= \int_{\mathbf{x} \in R_{\mathrm{lens}}(t)} \left(p f(\| \mathbf{x}\|) +(1-p)g(\| \mathbf{x}\|)\right)d\mathbf{x} +  \int_{\mathbf{x} \in R_{\mathrm{lune}}(t)}pf(\| \mathbf{x}\|) d\mathbf{x},\]
\begin{align}
Q(f,g)(t):= \int_{\mathbf{x} \in \mathbb{R}^2: \|\mathbf{x}\|\leq t}  p&\left( f(\|\mathbf{x}\|) -1 - f(\|\mathbf{x}\|)\log [f(\|\mathbf{x}\|)]\right) \notag \\
&+ (1-p)\left(g(\|\mathbf{x})\|) -1 - g(\|\mathbf{x}\|)\log [g(\|\mathbf{x}\|)]    \right)d\mathbf{x},\label{eq: weight def}\end{align}
and $q(f,g):=\lim_{t\rightarrow \infty} Q(f,g)(t)$, where $\|\cdot\|$ denotes the standard Euclidean $\ell_2$-norm.
\end{definition}
Note that for $f,g\geq 1$, we have $f-1-f\log f \leq 0$ and $g-1-g\log g \leq 0$. Thus $Q(f,g)$ is a non-increasing function of $t$ and either converges to a limit or to $-\infty$.
\begin{definition}[Symmetric local growth threshold]\label{def: theta sym local growth}
Let $Q_{\mathrm{max}}= Q_{\mathrm{max}}(p, \theta)$ denote the supremum of $q(f,g)$ over all continuous functions $f,g: \ \mathbb{R}_{\geq 0}\rightarrow \mathbb{R}_{\geq 1}$ satisfying
\begin{align*}
I(f,g)(t) > \theta \quad \forall t\in \mathbb{R}_{\geq 0}.
\end{align*}
Further, for $a>1$ and $p\in (0,1)$ fixed, let $\theta_{\mathrm{local}}=\theta_{\mathrm{local}}(a,p)$ be the supremum of all $\theta \leq 1$ such that
\begin{align}\label{eq: condition to guarantee local growth}
Q_{\mathrm{\max}}(p, \theta)> -1/a.
\end{align}
\end{definition}
\noindent Observe that $\theta_{\mathrm{local}}$ depends only on the values of $a$ and $p$, and that it is well-defined and greater or equal to $p$, since $Q_{\mathrm{\max}}(p,p)=0$, as can be seen by taking $f$ and $g$ to be constant and equal to $1$.
\begin{proof}[Proof of Theorem~\ref{theorem: threshold for symmetric local growth}]
Since $\theta$ satisfies \eqref{eq: symmetric growth threshold}, it follows from the definitions of $\theta_{\mathrm{local}}$ and $Q_{\mathrm{max}}$ that for any fixed $T\in \mathbb{R}_{\geq 0}$ there exist functions $f,g: \mathbb{R}_{\geq 0}\rightarrow \mathbb{R}_{\geq 1}$ and a real $\varepsilon>0$ such that the following hold:
\begin{enumerate}
	\item for every $t\in [0,T]$, $I(f,g)(t) > \theta+2\varepsilon$,
	\item $Q(f,g)(T+2)> -1/a +2\varepsilon$.
\end{enumerate}
We use this information to construct for any constant $C>0$ a $(C+2)r$-bounded event occurring w.h.p.\ and guaranteeing the existence of a ball of radius $Cr$ that becomes wholly infected over the course of our bootstrap percolation process. As we note in Proposition~\ref{proposition: inequality for theta local} below, this implies that $\theta_{\mathrm{local}}\leq (1+p)/2$ and so that for $\theta<\theta_{\mathrm{local}}$ the growing condition~\eqref{eq: growing condition} is satisfied. Combining this information with Theorem~\ref{theorem: (1+p)/2 threshold}(i) then yields the desired almost percolation.

We now give the details. Pick $T=(C+2)/\sqrt{\pi}$. Fix $\varepsilon>0$ and let $K>0$ be a sufficiently large positive real number to be specified later. Partition $\mathbb{R}^2$ into a fine grid of interior-disjoint $\frac{\sqrt{a\log n}}{K}\times \frac{\sqrt{a\log n}}{K}$ square tiles with axis-parallel sides. Let $\mathcal{F}$ be the finite family of tiles that lie wholly contained inside the ball of radius $(C+2)r= T\sqrt{a\log n}$ about the origin. Note that $\vert \mathcal{F}\vert =O(1)$. Let $f,g$ be functions $f,g: \mathbb{R}_{\geq 0}\rightarrow \mathbb{R}_{\geq 1}$ such that 1.\ and 2.\ above are satisfied.

Let $\mathcal{P}_p$ and $\mathcal{P}_{1-p}$ be independent Poisson point processes on $B_{T}(\mathbf{0})$ with intensities $p$ and $1-p$ respectively.  Let $E$ be the event that for every tile $F$ in $\mathcal{F}$, $F$ contains
\begin{itemize}
	\item at least $N_f(F):=\int_{\mathbf{x} \in F}  p f(\| \mathbf{x}\|_{2})d\mathbf{x}$ points of $\mathcal{P}_p$, and
	\item at least $N_g(F)\int_{\mathbf{x} \in F}  (1-p) g(\| \mathbf{x}\|_{2})d\mathbf{x}$ points of $\mathcal{P}_{1-p}$.
\end{itemize}
\begin{claim}\label{claim: prob of local start event} For any $\varepsilon>0$ fixed, there exists $K_1>0$ such that picking $K\geq K_1$ ensures
 \[\mathbb{P}(E)\geq e^{-c\log n +o(\log n)} \qquad \textrm{ for some fixed constant }c<1.\]	
\end{claim}	
\begin{proof}
Set $M_1$ to be the area of a $2$-dimensional ball of radius $T$.	 Since $f,g$ are continuous functions from the compact set $[0, T]$ to $\mathbb{R}_{\geq 1}$,  we have that $f$ and $g$ are bounded above by some $M_2>0$ on $[0,T]$ and further that $\log f$ and $\log g$ are uniformly continuous  over $[0,T]$. In particular, for every $\varepsilon>0$ there exists $K_1>0$ sufficiently large such that for all $K\geq K_1$, and all $x,y\in [0,T]$, $\vert x-y\vert \leq 2/K_1$ implies both of $\vert \log [f(x)]-\log[ f(y)]\vert$  and $\vert \log [g(x)]-\log[g(y)]\vert $ are less than $\varepsilon/M_1M_2$.

For each tile $F\in \mathcal{F}$, set $\rho_f(F):= N_f(F)/\vert F\vert$ and $\rho_g(F):= N_g(F)/\vert F\vert$. By the consequence of uniform continuity observed above, picking $K\geq K_1$ ensures that
\begin{align*}f(x)\left(\log [f(x)]- \log [\rho_f(F)]\right)\geq -\varepsilon /M_1\end{align*}
for all $x\in F$. Combining this with Lemma~\ref{lemma: Paul} (rescaled by a factor of $p$), the probability that $F$ contains at least $N_f(F)$ points of $\mathcal{P}_p$ is
\begin{align*}
\exp &\left\{ \left(\rho_f(F) -1- \rho_f(F)\log \left(\rho_f(F)\right) \right)p\vert F\vert +O(\log \log n) \right\}\\
&\geq \exp \left\{ \int_{\mathbf{x}\in F}p\left(f(\|\mathbf{x}\|) -1- f(\|\mathbf{x}\|)\log f(\|\mathbf{x}\|) \right)d\mathbf{x}  - p\frac{\varepsilon }{M_1}\vert F\vert +O(\log \log n) \right\}.
\end{align*}
One may obtain a similar expression  for the probability that $F$ contains at least $N_g(F)$ points of $\mathcal{P}_{1-p}$, substituting $1-p$ for $p$ and $g$ for $f$.

Since $\{\mathcal{P}_p \cap F, \ F\in \mathcal{F}\}$ and  $\{\mathcal{P}_{1-p} \cap F, \ F\in \mathcal{F}\}$ together form a collection of independent random variables, it follows that the probability that every tile $F\in \mathcal{F}$ contains at least $N_f(F)$ points of $\mathcal{P}_p $ and at least $N_g(F)$ points of $\mathcal{P}_{1-p}$ --- in other words, the probability of $E$ --- is at least
\begin{align*}
\mathbb{P}(E)&\geq  \exp \left\{ \left(Q(f,g)(T) -\varepsilon\right)a\log n +O(\log \log n) \right\},
   \end{align*}
which by Assumption 2.\ is  at least $e^{-(1-a\varepsilon)\log n +o(\log n)}$, proving our claim with the constant $c=1-a\varepsilon$.
\end{proof}
 Given an integer $i$, let $A_i=A_i(\mathcal{F})$ denote the collection of tiles of $\mathcal{F}$ that meet the annulus $B_{(i+1)\sqrt{a\log (n)}/K}(\mathbf{0})\setminus B_{{i}\sqrt{a\log (n)}/K}(\mathbf{0})$. Further, let $D_i=D_i(\mathcal{F})$ denote the union of the tiles of $\mathcal{F}$ that are wholly contained inside the disc $B_{{i}\sqrt{a\log (n)}/K}(\mathbf{0})$. Observe that both these tile families depend on $\mathcal{F}$ (and thus our choice of $K$).

\begin{claim} \label{claim: local growth infection spreads}
	For every $\varepsilon>0$ fixed, there exists $K_2>0$ such that picking $K\geq K_2$ ensures at least one of the following holds:
	\begin{itemize}
		\item the event $E$ fails to occur
		\item for every $i\leq KT-1$, every  tile $F\in A_i$ and every vertex $\mathbf{x}\in F$, we have that
\begin{align*}	
N_{\mathbf{x}}:=\vert B_{r}(\mathbf{x})\cap D_i\cap \mathcal{P}_{1-p}\vert + \vert  B_{r}(\mathbf{x})\cap \mathcal{P}_p\vert > \theta a \log n.
	\end{align*}
	\end{itemize}
\end{claim}
\begin{proof}
As in the previous claim, we note that the continuous functions $f,g$ are bounded above on the compact set $[0,T]$ by some $M_2>0$.

Assume the event $E$ occurs. Fix $i \leq KT-1$, and $F \in A_i$.  Let $\mathbf{x}\in F$ and set $t:= \|\mathbf{x}\|/\sqrt{a \log n}$. Denote by $S^1(\mathbf{x})$ the collection of tiles of $F$ wholly contained inside $D_i\cap B_{r}(\mathbf{x})$ and $S^2(\mathbf{x})$ the collection of tiles of $F$ wholly contained inside $B_r(\mathbf{x})$. Since $E$ occurs, we have
\begin{align*}
N_{\mathbf{x}}\geq \sum_{F\in S^2(\mathbf{x})} N_f(F) +\sum_{F\in S^1(\mathbf{x})}N_g(F).
\end{align*}
Now since $f,g$ are bounded above by $M_2$, the expression on the right hand side is at least
\begin{align*}
I(f, G)(t)a\log n  - M_2 p\left(a\log n   - \sum_{F\in S^2(\mathbf{x})}\vert F\vert\right) - M_2 (1-p)\left(\vert R_{\mathrm{lens}}\vert a\log n  -  \sum_{F\in S^1(\mathbf{x})}\vert F\vert\right).
\end{align*}
It readily follows from an application of Proposition~\ref{prop: bound on curve-meeting tiles} that for any $\varepsilon>0$, there exists $K_2>0$ such that picking $K\geq K_2$ ensures the above is at least $(I(f,g) - \varepsilon)a\log n$. Assumption 1.\ then tells us that
\begin{align*}
N_{\mathbf{x}}\geq (\theta+\varepsilon) a \log n,
	\end{align*}
	proving our claim.
\end{proof}
With Claims~\ref{claim: prob of local start event} and~\ref{claim: local growth infection spreads} in hand, the proof of Theorem~\ref{theorem: threshold for symmetric local growth} is now straightforward.

Now consider $(a, p, \theta)$ satisfying $a>1$ and~\eqref{eq: symmetric growth threshold}. As recorded in Proposition~\ref{proposition: inequality for theta local} below, this implies $\theta<(1+p)/2$. Let $C=C(a, \theta, p)>0$ be a constant such that by Theorem~\ref{theorem: (1+p)/2 threshold}(i), the infection of a ball of radius $Cr$ in $T_n$ leads w.h.p.\ to almost percolation in our bootstrap percolation process on $T_n$. This then specifies our choice of $T>0$.

Let $K_1, K_2$ be the constants from Claims~\ref{claim: prob of local start event}--\ref{claim: local growth infection spreads}. Pick $K\geq \max(K_1, K_2)$.  As remarked in Section~\ref{section: preliminaries}, by standard properties of Poisson point processes we may view the vertex set in the Bradonji\'c--Saniee model in the torus as being the union of a Poisson process $\mathcal{P}_p$ of initially infected points with intensity $p$  with a Poisson process $\mathcal{P}_{1-p}$ of initially uninfected points with intensity $1-p$.  Partitioning the torus $T_n$ into a fine grid of interior-disjoint square tiles with side-length $\sqrt{a \log n}/c'K$, where $c'\geq 1$ is chosen to ensure divisibility conditions are met, we can define a natural analogue $E'$ of our event $E$ inside $T_n$.

Clearly if $E'$ occurs, then by successive iterations of Claim~\ref{claim: local growth infection spreads}, all vertices inside a ball of radius $Cr$ about the origin in $T_n$ becomes infected, which then leads to almost percolation. Since $E'$ is $(C+2)r$-bounded and has probability $q=e^{-c\log n +o(\log n)}$ for some fixed constant $c<1$ (by Claim~\ref{claim: prob of local start event}), it follows from Lemma~\ref{lemma: first and second moment}(ii) that w.h.p.\ some translate of $E'$ occurs in $T_n$. Thus for $a>1$ fixed w.h.p.\ we have almost percolation in the $(p, \theta)$ regime satisfying~\eqref{eq: symmetric growth threshold}, as claimed.
\end{proof}
We conclude this subsection by recording some trivial inequalities between $\theta_{\mathrm{local}}$ and other quantities of interest, justifying the phase diagram we give in Figure~\ref{phase_d=2}.
\begin{proposition}\label{proposition: inequality for theta local}
		For all $a>1$ fixed and $p\in [0,1]$, the following inequalities hold:
		\[p\leq \theta_{\mathrm{local}} \leq \min \left(\frac{1+p}{2}, \theta_{\mathrm{start}}\right)\]
\end{proposition}
\begin{proof}
For the lower bound, observe that taking $f,g$ to be identically $1$ we have $I(f,g)\geq p$ for all $t\geq 0$ and $q(f,g)=0$. It then follows from, Definition~\ref{def: theta sym local growth} that $\theta_{\mathrm{local}}\geq p$ as claimed.

For the upper bound, observe first of all that if $p,\theta$ are fixed and $\theta >\theta_{\mathrm{start}}$ then w.h.p.\ no percolation occurs, which by Theorem~\ref{theorem: threshold for symmetric local growth} implies $\theta_{\mathrm{start}}\geq \theta_{\mathrm{local}}$. Next, note that the proof of Theorem~\ref{theorem: threshold for symmetric local growth} implies that for $a,p, \theta$ fixed with $a>1$ and $\theta <\theta_{\mathrm{local}}$, for any fixed $C>0$ w.h.p.\ there will be a ball of radius $C\sqrt{\log n}$ in $T_n$ such that every vertex inside it eventually becomes infected. On the other hand, Theorem~\ref{theorem: (1+p)/2 threshold}(ii) implies that if $(a,p, \theta)$ are fixed with $a>1$ and $\theta >\frac{1+p}{2}$, then w.h.p.\ there exists a constant $C>0$ such that no ball of radius $C\sqrt{\log n}$ in $T_n$ becomes infected in this way. This immediately yields that $\theta_{\mathrm{local}}\leq\frac{1+p}{2}$.
	\end{proof}
\begin{remark}\label{remark: derivatives of local curve}
The results in Appendix C show that $\sfrac{d}{dp}\left(\theta_{\mathrm{local}}(p)\right)\to\infty$ as $p\to 0$ and $\sfrac{d}{dp}\left(\theta_{\mathrm{local}}(p)\right)\to 1/2$ as $p\to 1$, so that $\theta_{\mathrm{local}}$ is tangent to $\theta_{\mathrm{start}}$ at $p=0$ and to $\theta=\frac{1+p}{2}$ at $p=1$. In particular the tangencies at $p=0$ and $p=1$ of the curve $\theta=\theta_{\mathrm{local}}(p)$ given in Figure~\ref{phase_d=2} are correct.
\end{remark}

		\subsection{Islands}\label{subsection: islands}
Beyond the almost-percolation guaranteed by Theorem~\ref{theorem: threshold for symmetric local growth}, it is natural to ask when full percolation occurs in the Bradonji\'c--Saniee model. One obstacle for full percolation is the presence of initially uninfected vertices of degree strictly less than $\theta a \log n$. Proposition~\ref{proposition: 0-stop} gives us the (implicit) $\theta$ threshold for the w.h.p.\ disappearance of such vertices. However there may be other, likelier, obstacles to full percolation, namely `islands' of initially uninfected vertices with few initially infected vertices and surrounded by sparsely populated `lakes' containing unusually few vertices.

Much as in the previous subsection, we are unable to rigorously determine the threshold for the w.h.p.\ disappearance of islands of uninfected vertices. We are however able to determine the threshold for the disappearance of symmetric islands, which is given by the solution of an explicit continuous optimisation problem and can be determined explicitly using the method of Lagrange multipliers.

As the arguments involved are very similar to those in the previous section, we give only a minimal level of detail. The basic idea is quite simple: we look for a radially symmetric distribution of infected/non-infected point densities which guarantees that some island cannot be infected from the outside. To such a distribution we associate a local $O(r)$ bounded event, from which we can in turn derive a threshold for the w.h.p.\ disappearance of such islands --- the proof essentially follows that of Theorem~\ref{theorem: threshold for symmetric local growth}, mutatis mutandis.

To make this more precise, let us give analogues of Definition~\ref{I and Q definitions} tailored to the island (rather than local outbreak) setting.
\begin{definition}
	Let $T>0$ be fixed.  For $t\in [0,T]$, write $R_{\mathrm{inner}}(t)$ for the asymmetric lens $B_{T}((0,0))\cap B_{1/\sqrt{\pi}}((t,0))$ and $R_{\mathrm{outer}}(t)$ for the (possibly empty) lune $B_{1/\sqrt{\pi}}((t,0))\setminus B_{T}((0,0))$.  A \emph{symmetric $T$-island distribution} is a pair $(f,g)$ of continuous functions $f,g: \ \mathbb{R}_{\geq 0}\rightarrow [0,1]$ such that for every $t\in [0,T]$, the following holds:
	\begin{align}\label{eq: island integral}
	\int_{\mathbf{x}\in R_{\mathrm{inner}}(t)} pf(\|\mathbf{x}\|) d_\mathbf{x} + 	\int_{\mathbf{x}\in R_{\mathrm{outer}}(t)} \left(pf(\|\mathbf{x}\|)+(1-p)g(\|\mathbf{x}\|)\right) d_\mathbf{x} < \theta.
	\end{align}
\end{definition}
\noindent Given a symmetric $T$-island distribution $(f,g)$, we define its \emph{weight} $q(f,g)$ as in~\eqref{eq: weight def}.
\begin{definition}[Symmetric islands threshold]
For every $T\geq 0$, the optimal $T$-island weight $q_{\mathrm{\max}}(T)=q_{\mathrm{\max}}(T)(p, \theta) $ is the supremum of the weight $q(f,g)$ over all symmetric $T$-island distributions $(f,g)$.  The \emph{symmetric islands threshold} $\theta_{\mathrm{islands}}=\theta_{\mathrm{islands}}(a,p)$ is then defined to be the supremum of the  $\theta \leq 1$ such that
\begin{align}\label{eq: theta islands def}
\sup_{T\geq 0} q_{\mathrm{max}}(T)>-1/a.
\end{align}
\end{definition}
\noindent We note that the quantity $\theta_{\mathrm{islands}}$ can in principle be computed using Euler--Lagrange equations --- see Appendix A. However the solution will be implicit rather than explicit. Let us record here however the simple fact that
\begin{align}\label{inequality for theta island}
\theta_{\mathrm{islands}}\leq \frac{1+p}{2}.
\end{align}
Indeed let $(a,p)$ and $\varepsilon>0$ be fixed.  Pick a constant $T>0$ sufficiently large such that  the area of $R_{\mathrm{inner}}(T)$ is at least $\frac{1}{2}-\varepsilon$. Then it is easily checked that taking the functions $f,g$ to be identically $1$ gives a symmetric $T$-island distribution $(f,g)$ with weight $q(f,g)=0>-1/a$ for any $\theta> \frac{1+p}{2}+\varepsilon (1-p)$. Inequality~\ref{inequality for theta island} follows immediately.
\begin{proof}[Proof of Theorem~\ref{theorem: islands}]
This is a simple modification of the proof of Theorem~\ref{theorem: threshold for symmetric local growth}: if $\theta> \theta_{\mathrm{islands}}$, then there exists a symmetric $T$-island distribution $(f,g)$ such that $q(f,g)>-1/a+2\varepsilon$. As in the proof of Theorem~\ref{theorem: threshold for symmetric local growth}, passing to a fine tiling of a ball of radius $(T+2)\sqrt{a\log n}$, one may use $(f,g)$ to define a tiling event $E$  in $T_n^2$ such that (i) $E$ is an $O(r)$-bounded event with probability $\mathbb{P}(E)= e^{-(c+o(1))\log n}$, where $c<1$ is a fixed constant, and (ii) if $E$ occurs then none of the points in a ball of radius $T\sqrt{a\log n}$ around the origin that are initially uninfected ever become infected in our bootstrap percolation process. Applying Lemma~\ref{lemma: first and second moment}(ii), we have that w.h.p.\ some translate of $E$ occurs, whence w.h.p.\ we do not have full percolation in this regime. We leave the details to the reader.
\end{proof}

The content of Conjecture~\ref{conjecture: islands} is that the last obstruction to full percolation that will vanish as we decrease the infection threshold $\theta$ will correspond to local distribution of initially infected/uninfected point that  are well-approximated by a symmetric $T$-island, for some $T=T(a,p)$. Motivation for the conjectured circular symmetry comes from the isoperimetric inequality in the plane as well as probabilistic considerations: for an uninfected island configuration to remain both uninfected and likely to occur in $T_n^2$, one expects it is best to minimise the length of its boundary, and to spread out unlikely low densities of points outside the island and of initially uninfected points inside the island as uniformly as possible.

Theorem~\ref{theorem: islands} gives an upper bound on the values of $\theta$ for which full percolation may occur w.h.p.\  in the Bradonji\'c--Saniee model. While we are unable to prove a matching lower bound and thereby prove Conjecture~\ref{conjecture: islands}, we note here that one can nevertheless prove some rigorous lower bounds on the threshold $\theta_{\mathrm{percolation}}(a,p)$ below which percolation occurs w.h.p., which we believe refine the earlier simple bounds due to Bradonji\'c--Saniee~\cite[Theorem 2]{BradonjicSaniee14} (in both cases, the exact value of the threshold is given implicitly rather than explicitly, which makes a comparison difficult).

Theorem~\ref{theorem: (1+p)/2 threshold}(ii) and~\ref{theorem: threshold for symmetric local growth} imply that once $\theta$ falls below $\theta_{\mathrm{local}}$, w.h.p.\ the only obstructions to full percolation are components of uninfected vertices of Euclidean diameter $O(\sqrt{\log n})$ in $T_n^2$. Let us consider what point configurations make the existence of such components possible. Given a component of never-infected points of diameter $C\sqrt{\log n}$, for some $C>0$, consider a pair of uninfected points $\mathbf{u}$, $\mathbf{v}$ from that component with $\|\mathbf{u}-\mathbf{v}\|=C\sqrt{\log n}$. Then the whole component of never-infected points lies inside the lune $L$ formed by the two discs of radius $C\sqrt{\log n}$ centred at $\mathbf{u}$ and $\mathbf{v}$ respectively. Since $\mathbf{u}$, $\mathbf{v}$  do not become infected, it must be the case that the number of points in $B_r(\mathbf{u})\setminus L$ plus the number of initially uninfected points in $B_r(\mathbf{u})\cap L$ is less than $\theta a \log n$, and a similar statement holds for $\mathbf{v}$. This implies that either $\left(B_r(\mathbf{u})\cup B_r(\mathbf{v})\right)\setminus L$ contains an abnormally low number of points from our Poisson point process, or that $\left(B_r(\mathbf{u})\cup B_r(\mathbf{v})\right)\cap L$ contains an abnormally low number of initially infected points. By performing a case analysis and some Lagrangian optimisation, one can upper-bound the probability of such an event.  Once this upper bound becomes $o(1/n)$, one can then apply Markov's inequality to show that such unlikely point configurations w.h.p.\ do not occur, and thus that we have w.h.p.\ entered the full percolation regime.  However we do not believe the bounds coming from this argument are optimal (since they will not match those from Conjecture~\ref{conjecture: islands}) or particularly helpful, so we relegate a sketch of the aforementioned case analysis and optimisation to Appendix B.

\section{Bootstrap percolation on the circle}\label{section: 1d}
In this section, we discuss the behaviour of the Bradonji\'c--Saniee model in dimension $d=1$, i.e.\ bootstrap percolation on a Gilbert random geometric graph on the circle. In the regime we are considering, this means we have a Poisson process of intensity $1$ on a circle of circumference $n$ providing us with the vertex-set of a Gilbert geometric graph with parameter $r$ given by $2r:=a\log n$, where $a>1$ is fixed.

As our main interest in this (already long) paper is the behaviour of the Bradonji\'c--Saniee model in the torus, we give a more informal, high-level discussion of the behaviour of the $1$-dimensional case, noting that all our arguments can be made fully rigorous using fine tiling arguments as in Section~\ref{section: 2d}.

The $1$-dimensional case reveals a broadly similar picture  to what we conjecture holds in the $2$-dimensional case, with two exceptions: (i) starting an infection is enough to ensure a large local outbreak, doing away with the necessity of computing an analogue of $\theta_{\mathrm{local}}$, and (ii)  the existence of \emph{blocking sets} (see below) that can stop a large local outbreak from becoming global, since in $1$-dimension you cannot `go round an obstacle'.  A pleasant feature of the $1$-dimensional case, however, is that the optimisation problems involved are far simpler and can be resolved explicitly using the method of Lagrange multipliers.

We break up our discussion of the $1$-dimensional model below in subsections on starting and growing an infection, blocking sets, and islands, before summarising the behaviour in the various regimes thus identified in Section~\ref{section: 1d summary} with a picture. Throughout this section, $(a,p,\theta)$ is fixed with $a>1$ and $0<p<\theta <1$.

\subsection{Growing an infection}

Our first observation is that when $f_{\rm start}(a,p,\theta)<1$, then, for some $\varepsilon$,
some vertex $v$ (at position 0, say) will see at least $a(\theta+\varepsilon)\log n$ infected neighbours, rather than just $a\theta\log n$.
This follows from the continuity of the function $f_{\rm start}$. We may also assume that both the infected and uninfected vertices in
$I_{100r}=(-100r,100r)$ are distributed uniformly, so that any interval $J$ of length $c\sqrt{\log n}$ inside $I_r$ contains
$|J|(1-p)(1+o(1))$ initially uninfected vertices and $|J|(\theta+\varepsilon)(1+o(1))$ initially infected ones. (The densities $1-p$
and $\theta+\varepsilon$ come from the colouring theorem for Poisson processes~\cite{Kingman93}, which implies in this case that the initially
infected and uninfected points can be regarded as {\it independent} Poisson processes of intensities $p$ and $1-p$ respectively. Meanwhile,
the uniformity assumption can be verified using Lemma~\ref{lemma: Paul}; the claimed uniformity fails with probability at most
$\log n\exp(-c'\sqrt{\log n})\to 0$, at any given place where the infection starts.)

One consequence of this is that not only does $v$ become infected (if it was not already infected), but that all vertices in the interval
$(-\delta\log n,\delta\log n)$ become infected, for some $\delta=\delta(\varepsilon)$. Using the notation $I_x=(-x,x)$, we see that, after
the first round of new infections, the interval $I_{\delta\log n}$ contains $|I_{\delta\log n}|(1-p+\theta+\varepsilon)(1+o(1))$ vertices,
all of which are infected, that $I_r\setminus I_{\delta\log n}$ contains $|I_r\setminus I_{\delta\log n}|(1-p)(1+o(1))$ uninfected vertices
and $|I_r\setminus I_{\delta\log n}|(\theta+\varepsilon)(1+o(1))$ infected ones, and that w.h.p.\ $I_{100r}\setminus I_r$ contains infected and
uninfected vertices with approximate densities $p$ and $1-p$ respectively.

Now we show that w.h.p.\ the `infected interval' $I_{\delta\log n}$ grows until it infects every vertex in $I_{100r}$. For simplicity, we
shall say that the interval $I_x$ is infected if every vertex within it is infected. Suppose that, after some rounds of the bootstrap
process, $I_x$ is in fact infected. (Initially, we may take $x=\delta\log n$.) We examine the neighbours of the first uninfected vertex $u$
to the right of $I_x$, with a view to showing that $u$ gets infected next, by virtue of having many infected neighbours in $I_x$. An identical
argument will apply on the left.

Assume first that $x\le r/2$. Then w.h.p.\ $u$ will have $N_1(x)(1+o(1))$ infected neighbours, where
\begin{align*}
N_1(x)&=(\theta+\varepsilon)(2r-3x)+xp+2x(1-p+\theta+\varepsilon)\\
&=2r\theta+\varepsilon(2r-x)+x(2-\theta-p)\\
&>2r\theta=a\theta\log n,
\end{align*}
so that $u$ does indeed become infected w.h.p..

Next assume that $r/2<x\le r$, and write $y=x-r/2$. This time w.h.p\ $u$ will have $N_2(y)(1+o(1))$ infected neighbours, where
\begin{align*}
N_2(y)&=(1-p+\theta+\varepsilon)r+\left(\frac{r}{2}-y\right)(\theta+\varepsilon)+\left(\frac{r}{2}+y\right)p\\
&=\frac{(2-p+3\theta+3\varepsilon)r}{2}+y(p-\theta-\varepsilon)\\
&>2r\theta=a\theta\log n.
\end{align*}
Now since $N_2(y)$ is a linear function satisfying
\[
N_2(0)=\frac{(2-p+3\theta+3\varepsilon)r}{2}>\frac{(1+3\theta)r}{2}>2r\theta
\]
and
\[
N_2\left(\frac{r}{2}\right)=\frac{(2+2\theta+2\varepsilon)r}{2}>2r\theta,
\]
it follows that $N_2(y)>2r\theta$ and hence that $u$ becomes infected w.h.p..

Finally we show that the infection spreads beyond $I_r$, to at least the entire interval $I_{100r}$. When $r<x\le 100r$, $u$ will have
at least $(r+rp)(1+o(1))$ infected neighbours (this will be an underestimate if $x<2r$), and this exceeds $2r\theta$ when $2\theta<1+p$.
Accordingly, we name the inequality $2\theta<1+p$ (which is already familiar to us from Theorem~\ref{theorem: (1+p)/2 threshold})
the {\it global growth condition} and display it for convenience.
\[
\boxed{{\rm\bf Global\ growth\ condition}\ \ \theta<\frac{1+p}{2}}
\]
We also have the condition from Proposition~\ref{prop: start} for the infection to start.
\[
\boxed{{\rm\bf Starting\ condition}\ \ a(p-\theta+\theta\log(\theta/p))<1\ \ {\rm or\ }\theta<p}
\]
Next, we describe the spread of the infection beyond $I_{100r}$. As a consequence of Lemma~\ref{lemma: first and second moment}(ii), if the starting condition is met, then the infection
will actually start in $n^{\alpha+o(1)}$ sites, where $\alpha= 1-f_{\mathrm{start}}(a,p,\theta)$. If the global growth condition is also met, these local infections will grow
on both sides from each such site, until one of them is met by a (left or right) {\it blocking set}, which we describe below.

\subsection{Blocking sets}
In this subsection, we assume that both the starting and the global growth conditions are both satisfied. What can then stop an infection from spreading around the circle?

A {\it clockwise}, or {\it right}, {\it blocking set} consists of two contiguous intervals, both of length $r$. The left interval contains
$zr$ vertices, and the right interval contains $xpr$ initially infected vertices. Even if all the vertices in the left interval become
infected, the infection will not spread to the right interval (or, more precisely, clockwise), as long as
\[
(z+xp)r<a\theta\log n=2r\theta,
\]
i.e.\  as long as $xp+z<2\theta$. We wish to maximize the probability that such a set occurs. Using Lemma~\ref{lemma: Paul}, we see that the probability $q=q(x,z)$ of a blocking set with
parameters $x$ and $z$ is given by
\begin{align*}
q&=\exp\{r(z-1-z\log z)+pr(x-1-x\log x)+o(r)\}\\
&=\exp\left\{\frac{a\log n}{2}\left((z-1-z\log z)+p(x-1-x\log x)+o(1)\right)\right\}.
\end{align*}
Consequently, to maximize the probability that a blocking set occurs, we must maximize
\[
f(x,y,z)=z-1-z\log z+p(x-1-x\log x)
\]
subject to the constraint
\[
g(x,z)=xp+z\leq 2\theta,
\]
with $p$ and $\theta$ fixed. Using the method of Lagrange multipliers, we see that $f(x,z)$ is maximized when
\[
(x,z)=\left(\frac{2\theta}{1+p},\frac{2\theta}{1+p}\right),
\]
and, since the growing condition is satisfied, we note that $x=z<1$ at the maximum.
Substituting, the maximum $f_{\rm max}$ of $f(x,z)$ subject to $g(x,z)\leq 2\theta$ is given by
\[
f_{\rm max}=2\theta-(1+p)-2\theta\log\left(\frac{2\theta}{1+p}\right),
\]
so that the maximum of $q(x,z)$ is
\[
q_{\rm max}=\exp\left\{-a\log n\left(\frac{1+p}{2}-\theta+\theta\log\left(\frac{2\theta}{1+p}\right)+o(1)\right)\right\}.
\]
We conclude that the expected number of blocking sets is $n^{\beta+o(1)}$, where
\[
\beta=1-a\left(\frac{1+p}{2}-\theta+\theta\log\left(\frac{2\theta}{1+p}\right)\right).
\]
Recall also that, as long as $\theta>p$, the infection starts in $n^{\alpha+o(1)}$ places, where
\[
\alpha=1-f_{\mathrm{start}}(a,p,\theta)=1-a(p-\theta+\theta\log(\theta/p)).
\]

Now assume that the starting and growing conditions are satisfied, and consider the case $0<\alpha<\beta<1$. Infections will
start to spread in $n^{\alpha+o(1)}$ places, and these will be blocked by $n^{\beta+o(1)}\gg n^{\alpha+o(1)}$ blocking
sets (a right blocking set stops infections spreading clockwise, and a left blocking set stops infections spreading
anticlockwise). Accordingly, each of the $n^{\alpha+o(1)}$ growing infections will spread for distance $n^{1-\beta+o(1)}$
before being blocked in both the clockwise and the anticlockwise direction, so that the infections will cover $n^{1+\alpha-\beta+o(1)}=o(n)$ of the circle. We call this the regime of `polynomial growth'; in this regime, the spread of infection is largely contained.

Next assume again that the starting and growing conditions are satisfied, but suppose that $0<\alpha/2<\beta<\alpha<1$. The
$n^{\alpha+o(1)}$ spreading infections will encounter $n^{\beta+o(1)}\ll n^{\alpha+o(1)}$ blocking sets, and, since
these blocking sets will usually be isolated, with spreading infections on each side, most of the circle will become infected.
The only obstructions (save for the `islands' to be described in the next subsection) are pairs of blocking sets (reading clockwise,
a right blocking set followed by a left blocking set) with no growing infection between them. There will be $n^{\beta}\cdot
n^{\beta-\alpha+o(1)}=n^{2\beta-\alpha+o(1)}$ such pairs, typically separated by distance $n^{1+\alpha-2\beta+o(1)}$, so that the entire
circle except for regions of length totalling $n^{1-2\beta+\alpha+o(1)}=o(n)$ will become infected. We call this the regime of
`polynomial obstructions': when $2\beta>\alpha$, blocking sets by themselves cannot prevent the spreading infections from covering
most of the circle.

There are thus two thresholds; the first, separating the regime of polynomial growth from that of polynomial obstructions, occurring
at $\alpha=\beta$, and the second, separating the regime of polynomial obstructions from that of `logarithmic obstructions' (caused
by {\it islands} -- see the subsection below), occurring at $\alpha=2\beta$.

When finding these thresholds, we recall that they lie entirely inside the region $\theta\ge p$, since $\alpha=1$ when $\theta\le p$.
For the first threshold, we have $\alpha>\beta$ exactly when
\[
\frac{1+p}{2}-\theta+\theta\log\left(\frac{2\theta}{1+p}\right)>p-\theta+\theta\log(\theta/p),
\]
which yields the following expression for the threshold.
\[
\boxed{{\rm\bf First\ threshold \ condition}\ \ \frac{1-p}{2}>\theta\log\left(\frac{1+p}{2p}\right)}
\]
Note that this region contains the region $\theta\le p$ (as can be seen from the Taylor expansion of the right hand side), and is independent of the choice of $a>1$. For the second threshold, we have $\alpha>2\beta$ exactly when
\[
1-a(p-\theta+\theta\log(\theta/p))>2\left(1-a\left(\frac{1+p}{2}-\theta+\theta\log\left(\frac{2\theta}{1+p}\right)\right)\right),
\]
which yields the following expression for the threshold.
\[
\boxed{{\rm\bf Second\ threshold \ condition}\ \ a\left(1-\theta+\theta\log\left(\frac{4\theta p}{(1+p)^2}\right)\right)>1\ \ {\rm or\ }\theta<p}
\]
This time, we need to include the additional condition $\theta<p$.

\subsection{Islands}
In this subsection, we assume that the starting, global growth, first and second threshold conditions are all satisfied. What can then prevent full percolation?

A $cr$-{\it island}, illustrated in Figure~\ref{blockingset}, consists of five contiguous intervals, whose lengths, from left to right, are
$cr,(1-c)r,cr,(1-c)r$ and $cr$. Inside these intervals, there are, from left to right, $zcr$ vertices, $y(1-c)r$ vertices,
$xpcr$ initially infected vertices, $y(1-c)r$ vertices and $zcr$ vertices. We note that even if all the vertices in the four outermost
intervals become infected, no vertices inside the middle interval will become infected, as long as
\[
cz+2y(1-c)+xpc<2\theta.
\]
We wish to maximize the probability that such a set occurs.
\begin{remark}\label{remark: symmetric distrib islands}
	Formally, we should consider islands where the proportion of vertices and infected vertices varies across the five intervals, i.e.\ where these contain $z_1cr$ vertices, $y_1(1-c)r$ vertices, $xpcr$ initially infected vertices, $y_2cr$ vertices and $z_2cr$ vertices respectively, and where we have two constraints, namely  $cz_1+(y_1+y_2)(1-c)+xpc<2\theta$ and $cz_2+(y_1+y_2)(1-c)+ xpc<2\theta$. However, it is easily checked that the probability of such configurations is maximised by symmetric configurations with $y_1=y_2$ and $z_1=z_2=z$, so that the $cr$-islands we consider are the only ones we need to worry about.
\end{remark}
Using Lemma~\ref{lemma: Paul}, we see that the probability $q=q(x,y,z,c)$ of an
island with parameters $x,y,z$ and $c$ is given by
\begin{align*}
q&=\exp\{2cr(z-1-z\log z)+2(1-c)r(y-1-y\log y)+pcr(x-1-x\log x)+o(r)\}\\
&=\exp\left\{\frac{a\log n}{2}\left(2c(z-1-z\log z)+2(1-c)(y-1-y\log y)+pc(x-1-x\log x)+o(1)\right)\right\}.
\end{align*}
Consequently, to maximize the probability that an island occurs, we must maximize
\[
f(x,y,z,c)=2c(z-1-z\log z)+2(1-c)(y-1-y\log y)+pc(x-1-x\log x)
\]
subject to the constraints
\[
g(x,y,z,c)=cz+2y(1-c)+xpc\leq2\theta,
\]
and $c\in [0,1]$ (for the island to exist and not to be a blocking set), with $p$ and $\theta$ fixed.

\begin{figure}[htp!]
\centering
\includegraphics[width=0.8\linewidth, trim = 0cm 0cm 0cm 4.5cm]{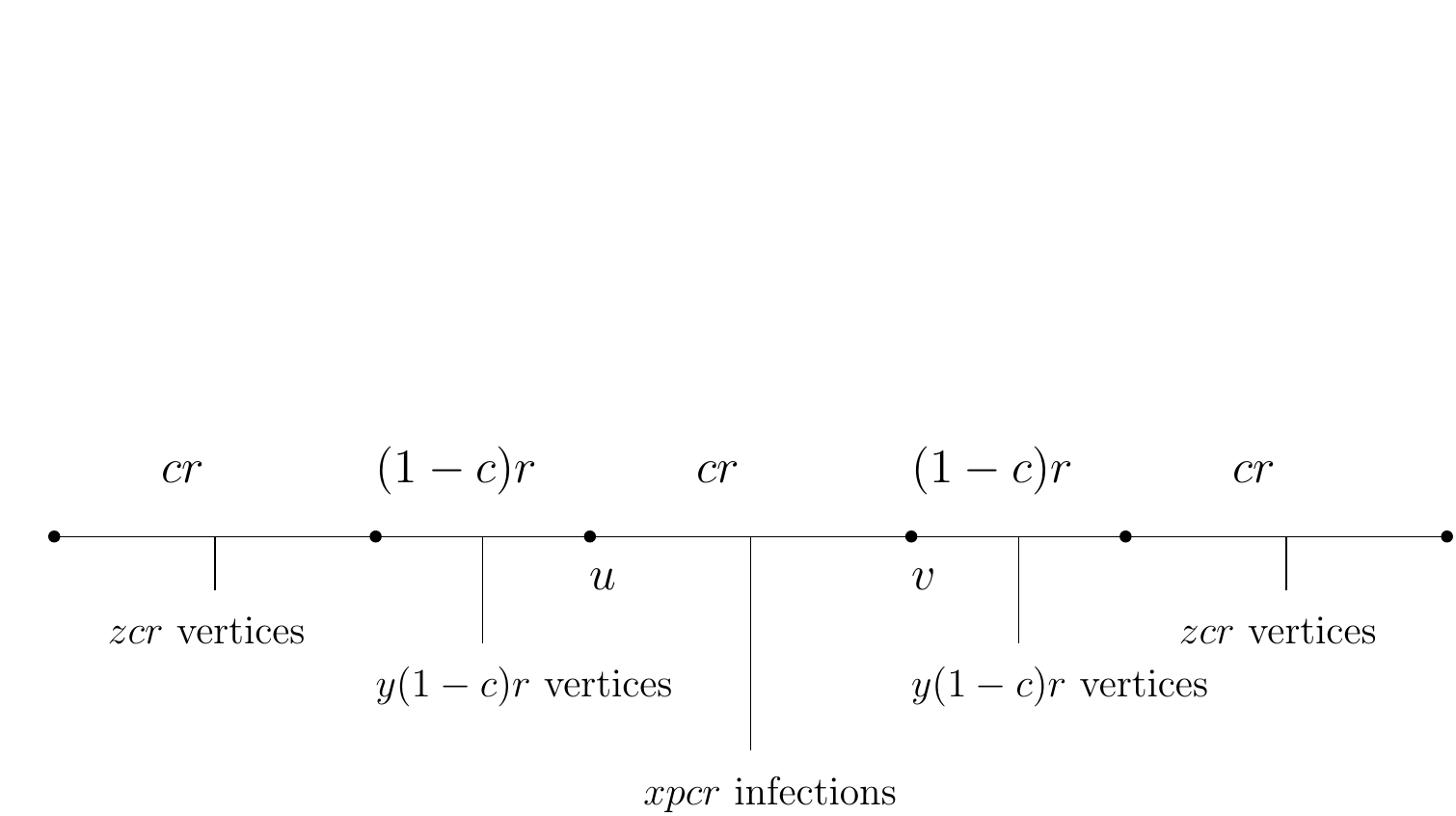}
\caption{\small An island}
\label{blockingset}
\end{figure}

Using the method of Lagrange multipliers, and noting that the growing condition is satisfied (which excludes the solution $x=y=z=1$),
we see that $f(x,y,z,c)$ is maximized when
\[
(x,y,z,c)=\left(\left(\frac{p}{2-p}\right)^2,\left(\frac{p}{2-p}\right)^2,\frac{p}{2-p},\frac{2(\theta(2-p)^2-p^2)}{p(1-p)(2-p)}\right),
\]
where we also require that $c\in[0,1]$ (for the island to exist). We separate the analysis into three cases, depending on the value of $c_{\star}=\frac{2(\theta(2-p)^2-p^2)}{p(1-p)(2-p)}$.

\noindent \textbf{Case 1: $c_{\star}\in(0,1)$.} The maximum $f_{\rm max}$ of $f(x,y,z,c)$ subject to our constrains is then given by
\[
f_{\rm max}=\frac{8(p-1)}{(2-p)^2}-4\theta\log\left(\frac{p}{2-p}\right),
\]
so that the maximum of $q(x,y,z,c)$ is
\[
q_{\rm max}=\exp{\left\{-a\log n\left(\frac{4(1-p)}{(2-p)^2}+2\theta\log\left(\frac{p}{2-p}\right)+o(1)\right)\right\}}.
\]
Set
\[
f_{\rm {c_\star}-stop}(a,p,\theta):=a\left(\frac{4(1-p)}{(2-p)^2}+2\theta\log\left(\frac{p}{2-p}\right)\right).
\]
By Lemma~\ref{lemma: first and second moment}(ii), we have that if $f_{\rm {c_\star}-stop}(a,p,\theta)<1$, then w.h.p.\ a $c_{\star}r$-island occurs somewhere on the circle, and we do not have full percolation. On the other hand if $f_{\rm c-stop}(a,p,\theta)>1$, then by Lemma~\ref{lemma: first and second moment}(i) w.h.p.\ there are no $cr$-islands for any $c$ and w.h.p.\ we have
full percolation (since the starting, growing and first and second threshold conditions are satisfied).

\noindent \textbf{Case 2: $c_{\star}\leq 0$.} When $c_{\star}\le 0$, the optimum legitimate island is the $0$-island, and we recover the necessary condition for full percolation from Proposition~\ref{proposition: 0-stop}.
namely that $f_{\rm 0-stop}(a,p,\theta)>1$. (This explains the earlier choice of notation.) Note that the analysis in this section reveals
that this necessary condition for full percolation $f_{\rm 0-stop}(a,p,\theta)>1$ is sufficient when $c_{\star}\le 0$, but not when $c_{\star}\in(0,1)$.
This is because, when $c_{\star}\in(0,1)$, the condition $f_{\rm 0-stop}(a,p,\theta)>1$ is strictly weaker than $f_{\rm c_{\star}-stop}(a,p,\theta)>1$,
since a non-degenerate $c_{\star}r$-island is more likely to occur than the degenerate $0$-island from Proposition~\ref{proposition: 0-stop}.

\noindent \textbf{Case 2: $c_{\star}\geq 1$.} When $c_{\star}\ge 1$, the optimum island has $c=1$, and a separate calculation with Lagrange multipliers shows that the optimum choices of
$x$ and $z$ ($y$ no longer features in the island) are
\[
(x,z)=\left(\frac{1+4\theta p-2\sqrt{1+8\theta p}}{2p^2},\frac{\sqrt{1+8\theta p}-1}{2p}\right)
=\left(\left(\frac{\sqrt{1+8\theta p}-1}{2p}\right)^2,\frac{\sqrt{1+8\theta p}-1}{2p}\right).
\]
These choices lead to a maximum of $q(x,y,z,c)$ of
\[
q_{\rm max}=\exp{\left\{-a\log n\left(1-\theta+\frac{p}{2}-\frac{\sqrt{1+8\theta p}-1}{4p}
+2\theta\log\left(\frac{\sqrt{1+8\theta p}-1}{2p}\right)+o(1)\right)\right\}},
\]
so that writing
\[
f_{\rm 1-stop}(a,p,\theta)=a\left(1-\theta+\frac{p}{2}-\frac{\sqrt{1+8\theta p}-1}{4p}
+2\theta\log\left(\frac{\sqrt{1+8\theta p}-1}{2p}\right)\right),
\]
we see that, when $c_{\star}\ge 1$, if $f_{\rm 1-stop}(a,p,\theta)<1$, then by Lemma~\ref{lemma: first and second moment}(ii), w.h.p.\ an $r$-island occurs somewhere on the circle, and we do not have full percolation. On the other hand if $f_{\rm 1-stop}(a,p,\theta)>1$, then by Lemma~\ref{lemma: first and second moment}(i) w.h.p.\ there are no $cr$-islands for any $c\in [0,1]$ and w.h.p.\ we have
full percolation (since the starting, growing and first and second threshold conditions are satisfied).

To summarize, suppose we are given $p$ and $\theta$, and we wish to check whether full percolation occurs, assuming that the
starting, growing and second threshold conditions are satisfied. First, we calculate
\[
c_{\star}=c_{\star}(p,\theta)"=\frac{2(\theta(2-p)^2-p^2)}{p(1-p)(2-p)}.
\]
Then, depending on the value of $c_{\star}$, the condition for full percolation is as displayed below.

\[
\boxed{{\rm\bf Full\ percolation}\ \ 1<
\begin{cases}
a(1-\theta+\theta\log\theta)&\text{\rm if $c_{\star}\le 0$}\\
a\left(\frac{4(1-p)}{(2-p)^2}+2\theta\log\left(\frac{p}{2-p}\right)\right)&\text{\rm if $0<c_{\star}<1$}\\
a\left(1-\theta+\frac{p}{2}-\frac{\sqrt{1+8\theta p}-1}{4p}
+2\theta\log\left(\frac{\sqrt{1+8\theta p}-1}{2p}\right)\right)&\text{\rm if $c_{\star}\ge 1$}\\
\end{cases}}
\]

The lack of islands is necessary for full percolation, and, given that the other four conditions are met, it is also sufficient.
To see this, suppose that the second threshold condition is satisfied. Then, once the infection has stopped spreading, the Gilbert
graph on the uninfected vertices splits into several components. Take one such component $C$, with extreme vertices $u$ and $v$.
If $u$ and $v$ lie at distance at least $2r$, we deduce the existence of a pair of blocking sets with no infection between them,
a situation excluded by the second threshold condition. Thus $u$ and $v$ lie at distance less than $2r$. If they lie at distance
less than $r$, then the likeliest configuration is that they are part of an island. If they lie at distance $(2-c)r$ with  $r\leq (2-c)r\leq 2r$, then we must have
five intervals of lengths $r,(1-c)r,cr,(1-c)r$ and $r$ (from left to right, with $u$ the rightmost endpoint of the first interval and $v$ the leftmost endpoint of the last interval), containing (respectively) $z_1r$ points,
$y_1p(1-c)r$ initially infected points, $xpcr$ initially infected points, $y_2p(1-c)r$ initially infected points and $z_2r$ points, where
\begin{align*}
z_1+(y_1+y_2)p(1-c)+xpc<2\theta && z_2+(y_1+y_2)p(1-c)+xpc<2\theta.
\end{align*}
It is easily checked that the probability of such a configuration is maximised by taking $y_1=y_2=y$ and $z_1=z_2=z$.  A short calculation with Lagrange multipliers then shows that the most likely way this can happen is when $c=0$ (corresponding to a blocking set) or $c=1$ (corresponding to an $r$-island).  In summary, if the second threshold condition is satisfied, then the most likely way for a small component of uninfected vertices to arise is in the shape of an island, and thus the threshold for
islands is also the threshold for full percolation.

\subsection{Summary}\label{section: 1d summary}

Figure~\ref{phase} is a phase diagram for the case $a=2$, showing the different regimes discussed above. If the starting condition is not
met, there is no growth -- no new infections take place. If the starting condition is met, but the global growth condition is not met,
then the growth will be confined to regions of width $\Theta(\log n)$ -- we term this `logarithmic growth'. If these two conditions,
but not the first threshold condition, are met, we are in the region of `polynomial growth'; if the first (but not the second)
threshold condition is met, then we have `polynomial obstructions'. Finally, if all these conditions are met, the threshold for
full percolation separates the region of `logarithmic obstructions' (islands) from that of full percolation. Note that the term
logarithmic (respectively, polynomial) obstruction refers to the size of the obstruction, rather than to the total proportion of
the circle occupied by such obstructions. It is possible that logarithmic obstructions will dominate polynomial ones for certain
values of the parameters, but, to keep things simple, we do not pursue this question further here.

It is convenient to visualize the spread of infection as a continuous process, starting in certain random places, and then spreading
in the form of expanding arcs, before possibly being blocked by blocking sets or islands. Indeed, consider a typical `interface' $I$
at the place where an infection is blocked. To the left of $I$, we have full infection, while to the right of $I$, only initially
infected points are infected. The interface $I$ itself consists of uninfected points alternating with (initially infected points
and) newly infected points. Now the probability that the length $d_I$ of $I$ exceeds $x$ decays exponentially in $x$, so long
interfaces are comparatively rare. Thus they may be conveniently visualized as single points marking the boundaries of infections,
and the frozen state of the model may be visualized as a collection of random arcs on a circle.

\begin{figure}[htp!]
\centering
\includegraphics[width=\linewidth]{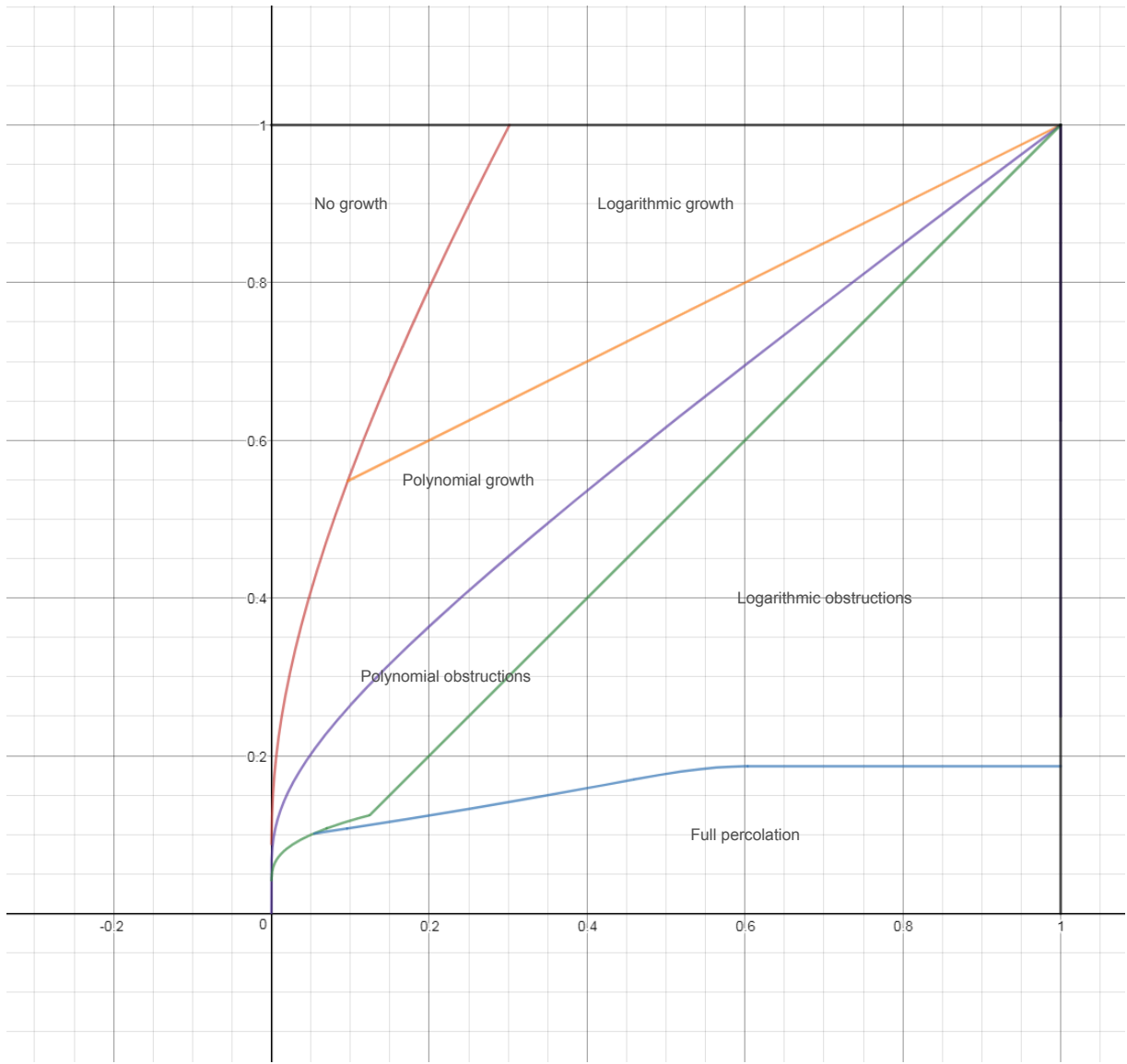}
\caption{\small Phase diagram for $a=2$}
\label{phase}
\end{figure}
	
\section{Concluding remarks}\label{section: concluding remarks}
We end this paper with some remarks.
	\begin{itemize}
		\item We stated and proved our main results for the Bradonji\'c--Saniee model on the $2$-dimensional torus $T_n:=T_n^2$ rather than the square $S_n:=[0,\sqrt{n}]^2$ to avoid having to consider boundary effects. We note here that the boundary effects do not play any significant role in the Bradonji\'c--Saniee model on $S_n$, apart from complicating the analysis relative to $T_n$.

		Indeed, since the presence of a boundary does not help the infection, the analogue of Theorem~\ref{theorem: (1+p)/2 threshold}(ii) immediately follows for $S_n$.

		For the analogue of Theorem~\ref{theorem: (1+p)/2 threshold}(i) in the square, one must modify the proof. Instead of considering the whole of $S_n$, one should instead focus on the subsquare $S_n'$ consisting of all rough tiles at graph distance at least $7$ from a `boundary tile' in the auxiliary graph $H$, and one must employ the Bollob\'as--Leader edge-isoperimetric inequality for the square grid rather than the toroidal grid. Since $o(\vert \mathcal{R}\vert)$ rough tiles lie close to the boundary of $S_n$, this results in some changed constants as well as the presence of a connected component in $G_{n,r}[\mathcal{P}\setminus A_{\infty}]$ of diameter $\Omega(\sqrt{n})$ and order $O(\sqrt{n\log n})$ all around the boundary of the square. Thus the bound on the Euclidean diameter of components of forever-uninfected points will only hold for components all of whose points are at a sufficiently large multiple of $\sqrt{\log n}$ away from the boundary.

\item We expect that Theorem~\ref{theorem: (1+p)/2 threshold} also holds in dimension $d\geq 3$. However some work will be required to adapt our arguments to the higher dimensional setting. For instance, we note that the proof of
Proposition~\ref{prop: bound on curve-meeting tiles}, and both the statement and the proof of Lemma~\ref{lemma: large boundary implies large white comp in H7} (involving dual cycles) use $2$-dimensional ideas and would need to be handled differently in dimension $d\ge 3$.

		\item In contrast with the constant infection threshold work of Candellero and Fountoulakis~\cite{CandelleroFountoulakis16} and on hyperbolic random geometric graphs and of Koch and Lengler~\cite{KochLengler21} on inhomogeneous random graphs, the behaviour for the Bradonji\'c--Saniee model appears to be rather different: for $p<\theta$ it is possible for some vertices to be infected in the first round of the process but for the process to only spread the infection to $\vert A_{\infty}\setminus A_0\vert =o(n)$ vertices. Indeed, if $\frac{1+p}{2}<\theta$ and  $f_{\mathrm{start}}<1$ both hold, this occurs w.h.p.\ as a consequence of Theorem~\ref{theorem: (1+p)/2 threshold}(ii) and Proposition~\ref{prop: start}. Both of these conditions can be satisfied simultaneously provided we pick $a>1$ sufficiently small and $p<1$ sufficiently large.
	\end{itemize}
Our work also leaves a number of questions open.
\begin{itemize}
	\item
Foremost among these is the problem of proving Conjecture~\ref{conjecture: either almost no percolation or almost percolation}. Our intuition is that one may need to consider far more subtle tile colourings to achieve this. Explicitly, rather than colour a fine tile $T$ red if all its points eventually become infected, one should plausibly instead fix some large integer constant $Q=Q(K)$ and label $T$ with $(i_T/Q, j_T/Q)$ if it contains between $\frac{i_T}{Q}\vert T\vert$ and
$\frac{i_T+1}{Q}\vert T\vert$ points of $A_{\infty}$, and between $\frac{j_T}{Q}\vert T\vert$ and
$\frac{j_T+1}{Q}\vert T\vert$ points of $\mathcal{P}\setminus A_{\infty}$, where $i_T, j_T$ are non-negative integers. A similar idea was used in~\cite{BalisterBollobasSarkarWalters09}. One would then need to show that for most fine tiles, $(i_T/Q, j_T/Q)$ is close to either $(1,0)$ or to $(p,1-p)$, and to analyse the component structure of fine tiles $T$ with $(i_T/Q, j_T/Q)$ close to $(1,0)$, deploying more refined versions of the arguments from the proof Theorem~\ref{theorem: giant spreads} to show such tiles either form an overwhelming majority or an overwhelming minority of tiles.

\item Another intriguing problem left open by Theorem~\ref{theorem: (1+p)/2 threshold} is what happens if $\theta=\frac{1+p}{2}\pm o(1)$  and we adversarially infect a ball of radius $Cr$? Here it is not clear to us what kind of behaviour one should expect. Could it be for example that both $\vert A_{\infty}\setminus A_0\vert$ and $\vert\mathcal{P}\setminus A_{\infty}\vert$ have order $\Omega(n)$ ? The analysis required to understand the typical behaviour in this regime is likely to be delicate.

\item Given that bootstrap percolation has been studied on $\mathbb{Z}^d$, it would be natural to consider bootstrap percolation on  a supercritical Gilbert disc graph in the plane, i.e.\ on the host graph $G_r(\mathbb{R}^2)$, where $\pi r^2$ is a sufficiently large constant (to ensure $G_r(\mathbb{R}^2)$ almost surely contains an infinite connected component), and the infection threshold is some constant $T$. For what $p\geq 0$ does an initial infection probability guarantee the almost sure emergence of an infinite connected component of eventually infected vertices?

\item It would also be natural to study an analogue of the Bradonji\'c--Saniee model on the $k$-nearest neighbour random geometric graph model, or on models of random geometric graphs in the torus allowing for the presence of some long distance edges by superimposing e.g.\ a sparse Erd{\H o}s--R\'enyi random graph or a configuration model on top of the Gilbert random geometric graph. 

\item Conjectures~\ref{conjecture: symmetric growth} and~\ref{conjecture: islands} provide an obvious area where there is considerable room for improvement on the results of the present paper. Progress on these fronts may require a better understanding of the behaviour of the solutions to the optimisation problems used to define the thresholds $\theta_{\mathrm{local}}$ and $\theta_{\mathrm{islands}}$ (our conjectured thresholds for almost percolation and full percolation respectively).

One question of particular interest to us is whether, given a fixed  triple $(a,p, \theta)$, one can identify the `critical radius' for local infections or islands. To be more precise, we expect that there may be a constant $C> 0$ such that spreading a local infection to radius $Cr$ is `harder' (less likely) than both spreading it to a radius $(C-\varepsilon )r$ and spreading an infection from a ball of radius $Cr$ to a ball of radius $(C+\varepsilon)r$. Our heuristic is that while the local outbreak is small it requires unlikely point configurations to spread radially outwards, but that once it gets sufficiently large then `global' behaviour kicks in and the infection is carried forward by its momentum without requiring a high density of infected points near its boundary.

A motivation for determining such a critical radius would be the possibility of explicitly determining and computing $\theta_{\mathrm{local}}$. Similarly, we expect that there is an optimal radius for islands resisting full percolation, and determining that optimal radius would help give a more explicit form for $\theta_{\mathrm{islands}}$.

\item Finally, given the motivations for studying bootstrap percolation, it would be natural to consider variants of the Bradonji\'c--Saniee model where e.g.\ some vertices are vaccinated or have a higher threshold for infection. See for example~\cite{EinarssonLenglerMoussetPanagiotouSteger19} for some recent work in this vein.

\end{itemize}

\section*{Acknowledgements}
Research on this project was done while the second author visited the first author in Ume{\aa} in Spring 2018 with financial support from STINT Initiation grant  IB 2017-7360, which the authors gratefully acknowledge.

\section*{Appendix A: analysis of $\theta_{\mathrm{islands}}$ via the Euler--Lagrange equations}
		\begin{figure}[htp!]
			\centering
			\includegraphics[width=0.5\linewidth]{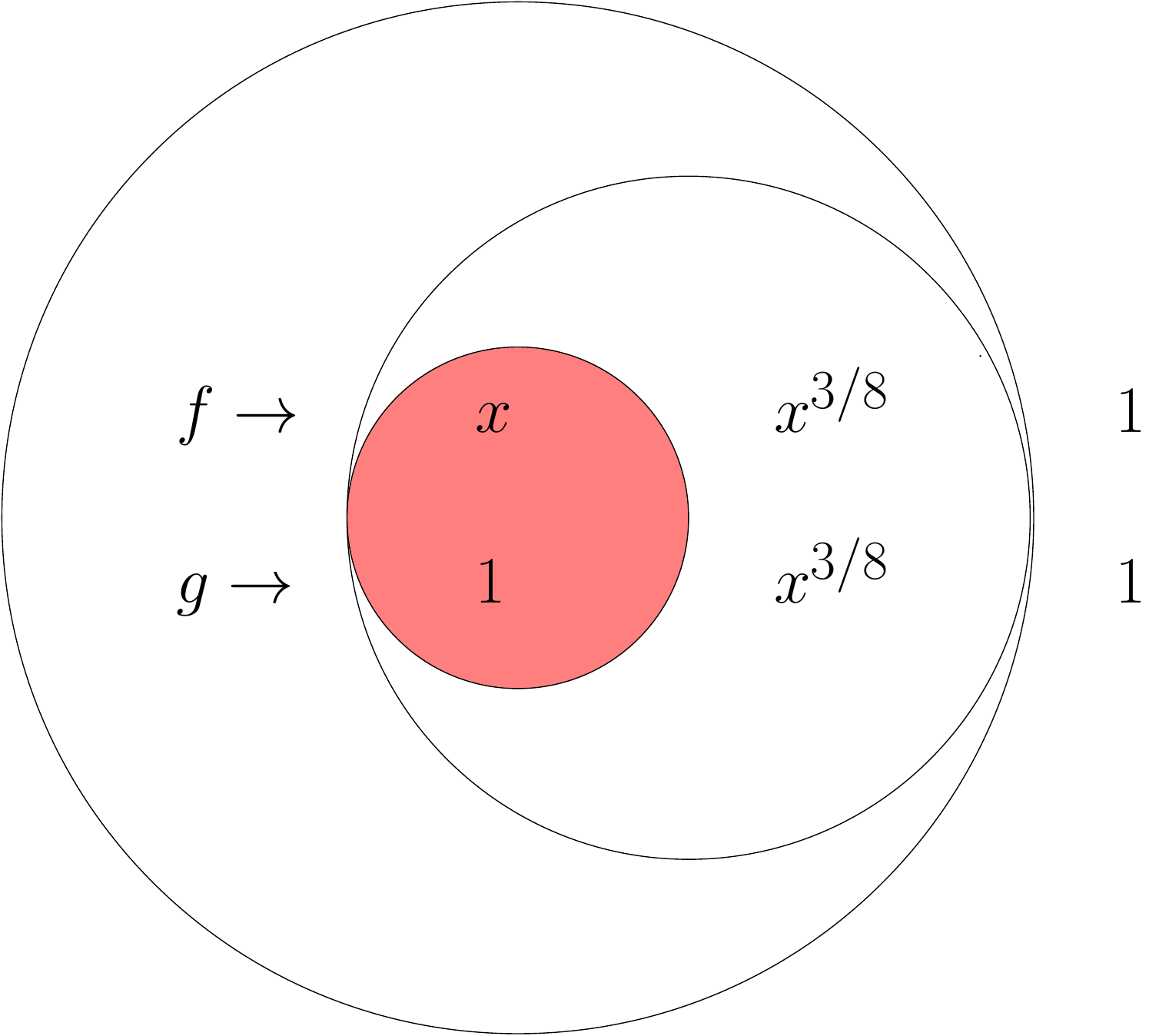}
			\caption{\small A simple obstruction to full percolation. Here $f$ and $g$ denote the density of initially and initially infected points in the various regions. The centre of the largest circle is $\mathbf{0}$ and $\mathbf{u}$ sits on the boundary of the shaded red disc of radius $\tau$ about $\mathbf{0}$; circles of of radius $t$ about $\mathbf{0}$ and $1$ about $\mathbf{u}$ are represented in the picture.}
			\label{simple}
		\end{figure}
		Recall that $\theta_{\mathrm{islands}}=\theta_{\mathrm{islands}}(a,p)$ is the threshold for the appearance of symmetric islands. In this appendix, we bound this threshold using the Euler-Lagrange equations, applied to~\eqref{eq: island integral}.

		As a warm-up, consider the following island, illustrated in Figure~\ref{simple}.
		We consider two concentric circles of radii $r/2$ and $3r/2$,
		forming a disc and an annulus of areas $A_x$ and $A_z$ respectively. Suppose that the inner disc contains $xpA_x$ initially infected points,
		and that the annulus contains $zA_z$ points. Then, even if all the points in the annulus become infected, no points in the
		inner disc become infected, as long as
		\[
		G(x,z)=3z+px<4\theta.
		\]	
The probability of the configuration is
		\[
		q=\exp\left\{\frac{a\log n}{4}F(x,z)\right\},
		\]
		where
		\[
		F(x,z)=8(z-1-z\log z)+p(x-1-x\log x),
		\]
		and so we must maximize $F(x,z)$ subject to $G(x,z)=4\theta$. A short
		calculation with Lagrange multipliers shows that $z=x^{3/8}$, while $x$ is determined by the equation
		\[
		3x^{3/8}+px=4\theta.
		\]
		Once $x$ has been determined, the threshold is given by
		\[
		4+a\left\{8(x^{3/8}-1-\tfrac{3}{8}x^{3/8}\log x)+p(x-1-x\log x)\right\}=0.
		\]
The lack of such islands is a necessary condition for full percolation. To refine this condition, we will consider a radially symmetric island, with
continuously varying densities of infected and uninfected points, as in Section~\ref{subsection: islands}. Such an island generalizes the one just considered.
To simplify the exposition and calculations, we will use a slightly different normalization (scaling by $r$ rather than $\sqrt{a\log n}$) and notation from that in Section~\ref{subsection: islands}.

Specifically, consider the following island, centred at $\mathbf{0}$. At distance $tr$ from $\mathbf{0}$, the densities of infected and uninfected points
are $pf(t)$ and $(1-p)g(t)$ respectively, where $f$ and $g$ satisfy $f(t)\to 1$ and $g(t)\to 1$ as $t\to\infty$. Rescaling by $r$,
for $\mathbf{u} \in\partial B_{t}(\mathbf{0})$, write
\[
A(t)=B_{1}(\mathbf{u})\cap B_{t}(\mathbf{0}){\rm \ \ \ and\ \ \ } B(t)=B_{1}(\mathbf{u})\setminus B_{t}(\mathbf{0}),
\]
exactly as before. Then the condition for an infection to stop spreading in towards $\mathbf{0}$ around the circle $\partial B_{\tau r}(\mathbf{0})$ is that
\[
p\int_{A(\tau)\cup B(\tau)}f\,dA+(1-p)\int_{B(\tau)}g\,dA< \pi\theta.
\]

At this point we need to make some simplifying assumptions, which, while restricting the dimensions of the island (and thus possibly rendering
it suboptimal), greatly facilitate calculations. First, we will assume that $\tau\le 1$. Second, we will assume that the optimal function $f$
is increasing for $t\ge 0$,
and that the optimal $g$ is increasing for $t\ge\tau$ (an assumption which will be consistent with the solution we obtain).
The reason for these assumptions
is that, without them, points in the interior of $B_{\tau r}(\mathbf{0})$ might see more infected neighbours than points on the boundary $\partial B_{\tau r}(\mathbf{0})$.
It is then conceivable that the infection could spread from the inside of the island outwards, causing the entire island to succumb to infection.

Given these assumptions, we must maximize
\[
q(f,g)=p\int(f-1-f\log f)\,dA+(1-p)\int(g-1-g\log g)\,dA
\]
subject to the above constraint, for fixed $\tau\in[0,1]$. The sets $A(\tau)$ and $B(\tau)$, as well as the optimal solutions $f$ and $g$,
are illustrated in Figure~\ref{small} (for $\tau\le1/2$) and Figure~\ref{big} (for $1/2\le\tau\le 1$).

\begin{figure}[htp!]
	\centering
	\includegraphics[width=\linewidth]{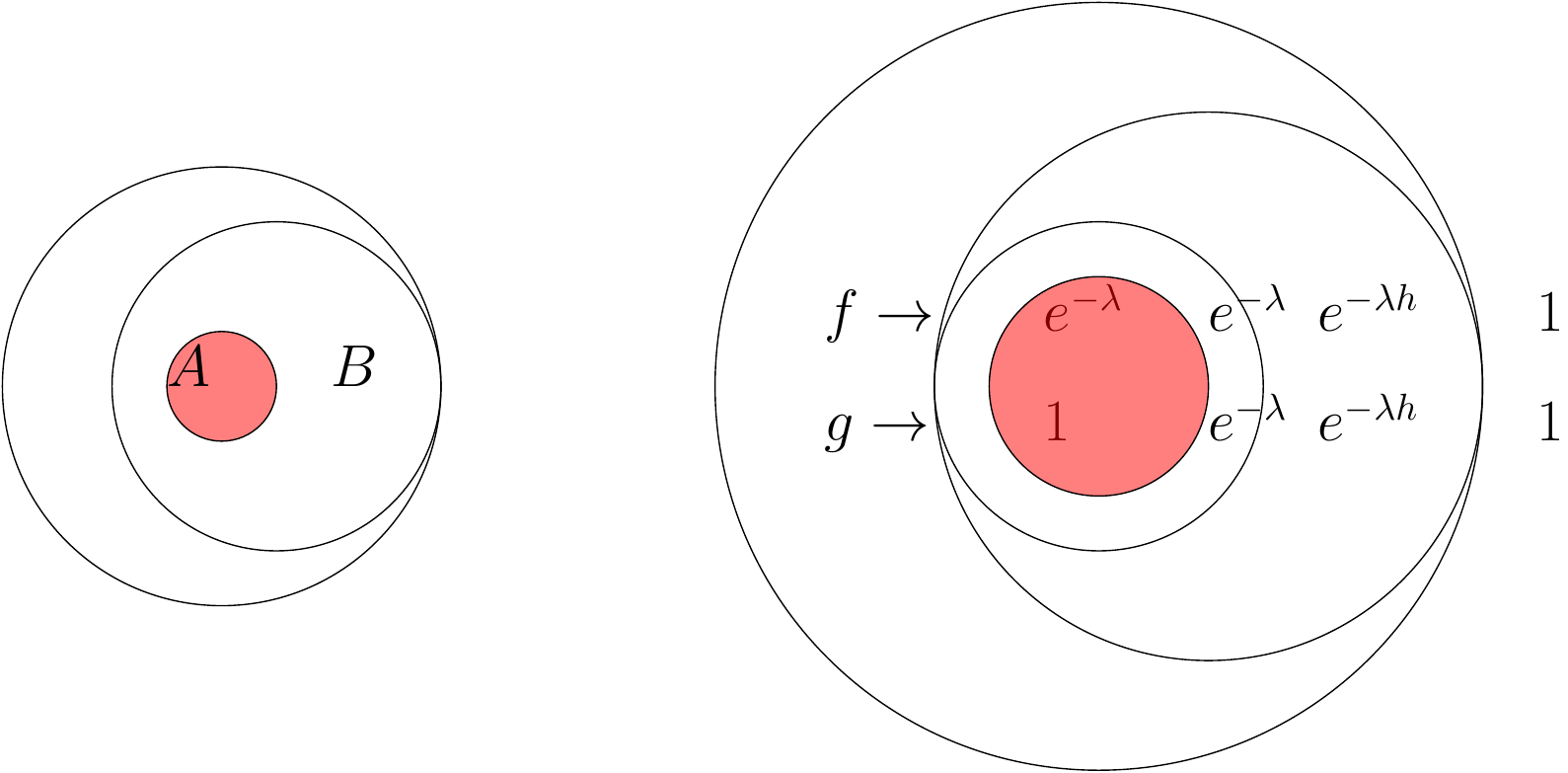}
	\caption{\small An obstruction with $\tau\le1/2$}
	\label{small}
\end{figure}

\begin{figure}[htp!]
	\centering
	\includegraphics[width=\linewidth]{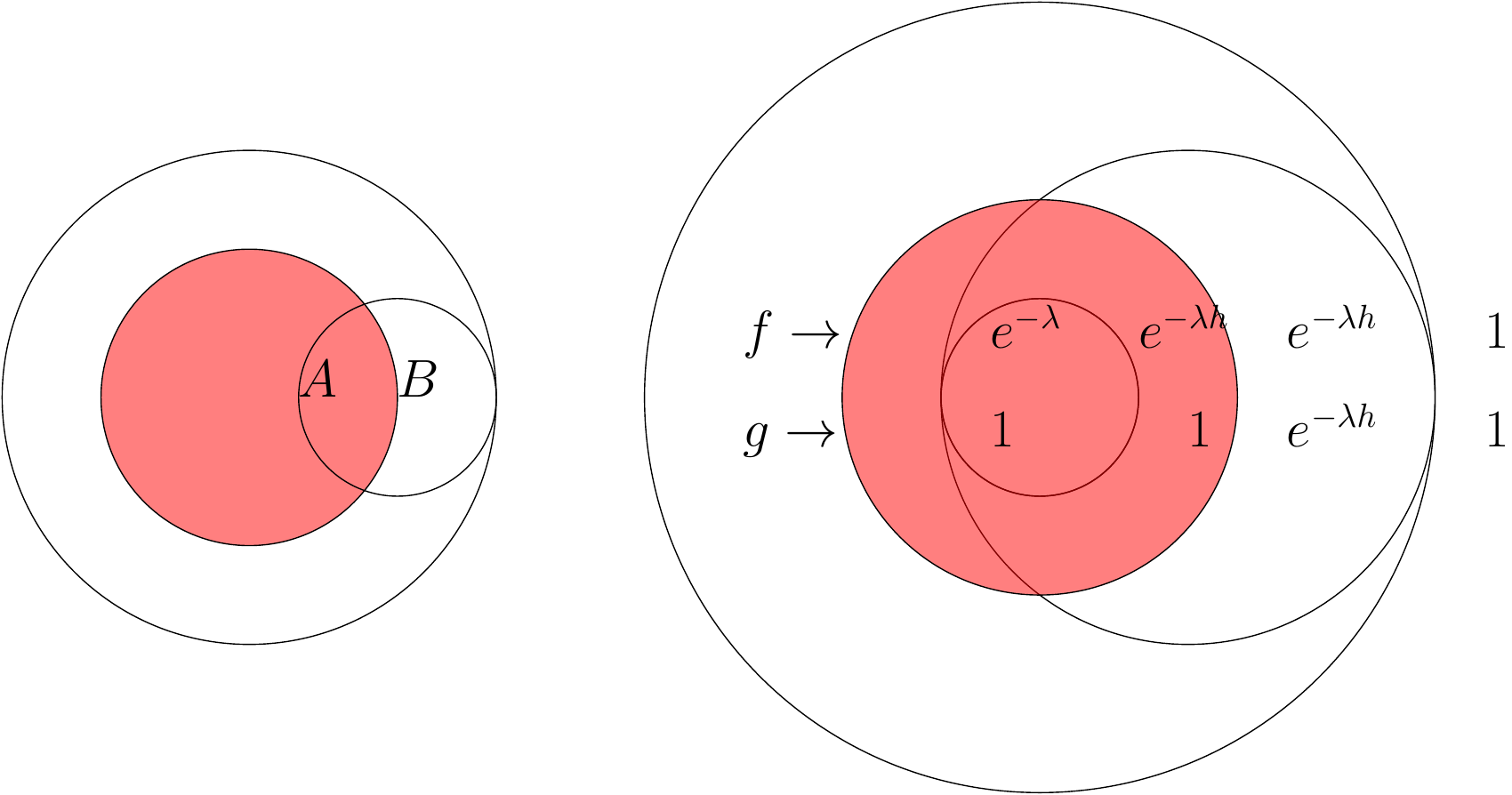}
	\caption{\small An obstruction with $1/2\le\tau\le1$}
	\label{big}
\end{figure}

Using the method of Lagrange multipliers, we maximize
\begin{align*}
L(f,g)= p\int(f-1-f\log f)\,dA &+(1-p)\int(g-1-g\log g)\,dA\\
&-\lambda\left(p\int_{A(\tau)\cup B(\tau)}f\,dA+(1-p)\int_{B(\tau)}g\,dA\right).
\end{align*}
To evaluate the last two integrals, we introduce the function $h(t)$, defined by the equation
\[
h(t)=
\begin{cases}
1&\textrm{if }t\leq 1-\tau\\
\tfrac{1}{\pi}\cos^{-1}\left(\frac{t^2-\tau^2-1}{2\tau}\right)&\textrm{if }1-\tau< t\le 1+\tau \\
0&\textrm{if }1+\tau < t.\\
\end{cases}
\]
The function $h(t)$ is the proportion of the ($1$-dimensional) circle of radius $t$, centred at $\mathbf{0}$, that lies in the ($2$-dimensional) ball $B_{1}(\mathbf{u})$, where $\mathbf{u}\in\partial B_{\tau}(0)$. 
Using this notation, we have
\begin{align*}
(2\pi)^{-1}L(f,g)
&=p\int_0^{\infty}f(t)-1-f(t)\log [f(t)]\,dt+(1-p)\int_0^{\infty}g(t)-1-g(t)\log [g(t)]\,dt\\
&\,\,-\lambda\left(p\int_{0}^{\infty}f(t)h(t)\,dt+(1-p)\int_{\tau}^{\infty}g(t)h(t)\,dt\right)\\
&=p\int_0^{\infty}f(t)-1-f\log [f(t)]\,dt+(1-p)\int_0^{\infty}g(t)-1-g(t)\log [g(t)]\,dt\\
&\,\,-\lambda\left(p\int_0^{1-\tau}f(t)\,dt+p\int_{1-\tau}^{1+\tau}f(t)h(t)\,dt+(1-p)\int_{\tau}^{1+\tau}g(t)h(t)\,dt\right)\\
&=\int_0^{\infty}F(f,g,t)\,dt.
\end{align*}
The Euler--Lagrange equations reduce in this case to $\partial F/\partial f=0$ and $\partial F/\partial g=0$.
These equations have the solution
\[
f(t)=
\begin{cases}
\exp(-\lambda)&t\le 1-\tau\\
\exp(-\lambda h(t))&1-\tau\le t\le 1+\tau\\
1&t\ge 1+\tau,\\
\end{cases}
\]
and
\[
g(t)=
\begin{cases}
1&t\le\tau\\
\exp(-\lambda)&\tau\le t\le \max(\tau,1-\tau)\\
\exp(-\lambda h(t))&\max(\tau,1-\tau)\le t\le 1+\tau\\
1&t\ge 1+\tau,\\
\end{cases}
\]
where $\lambda$ is determined from the constraint
\[
p\int_{0}^{\infty}2\pi f(t)h(t)\,dt+(1-p)\int_{\tau}^{\infty}2\pi g(t)h(t)\,dt=\pi\theta.
\]
For a given $\tau$, this allows us to compute $q_{\mathrm{max}}(\tau)$. Finally, we optimize over $\tau$, and set
\[
\sup_{\tau} aq_{\rm max}(\tau)=-\pi
\]
to determine the $\theta$-value of the threshold for the disappearance of islands. Note that the warm-up example (in which $\tau=1/2$) is in some sense a step-function approximation to this solution. As $p\to1$, we expect that the optimal value of $\tau$ tends to zero.

Note the analysis here provides an alternative derivation of the islands in the case $d=1$. We need only replace $h(t)$ by its one-dimensional version $h_1(t)$ given by, 
\begin{align*}
h_1(t)=
\begin{cases}
1& \textrm{if } \tau\leq 1-t\\
1/2& \textrm{if } 1-\tau <  t \leq 1+\tau\\
0& \textrm{if }  1+\tau< t.\\
\end{cases}
\end{align*}
\section*{Appendix B: lower bounds for the threshold for full percolation}
In this appendix, we sketch out some more details of the case analysis and Lagrangian optimisation that can be used to give rigorous (but almost certainly non-optimal) lower bound on the full percolation threshold, continuing the discussion at the end of Section~\ref{subsection: islands} (with the same notation).

Draw discs of radius $r$ around $\mathbf{u}$ and $\mathbf{v}$, resulting in one of the configurations in Figure~\ref{stopping}, according to whether $\delta:=C\sqrt{\log n}/r$ lies in the range $(0,1],(1,2]$ or $(2,\infty)$. Here $x$ and $zp$ denote the density of points and initially infected points of the process in the corresponding region respectively, and similarly $y$/$py$ denote the density of points/initially infected points in the corresponding region. We know the points $\mathbf{u}$ and $\mathbf{v}$ do not become infected as part of the bootstrap percolation process, even though all points outside the lune $L$ are infected. Our aim is to find the likeliest set of point densities in the appropriate regions making this event possible. We will then calculate the probabilities
of these likeliest configurations, and deduce that if the fixed triple $(a,p, \theta)$ is such that all such configurations have probability $o(1/n)$, no island can exist w.h.p..


Assuming that the densities of infected points in the various regions are as indicated, the most likely obstructions can be identified by
optimizing $x,y,z$ and $\delta$ for each case. Write $L(\delta)$ for the area of the lune formed by two unit discs whose centres lie at distance $\delta$, so that
\[
L(\delta)=\pi-\gamma-\sin\gamma,
\]
where $\gamma$ is given by
\[
\frac{\delta}{2}=\sin\left(\frac{\gamma}{2}\right).
\]
Also, write $M(\delta)$ for the area of the lune formed by a disc $B_{\delta}$ of radius $\delta$ and a unit disc $D$ whose center lies on the
perimeter of $B_{\delta}$, so that
\[
M(\delta)=\delta^2(\pi-\beta-\sin\beta)+\frac{\beta}{2},
\]
where $\beta$ is given by
\[
2\delta=\sec\left(\frac{\beta}{2}\right).
\]

\begin{figure}[htp!]
\centering
\includegraphics[width=0.85\linewidth]{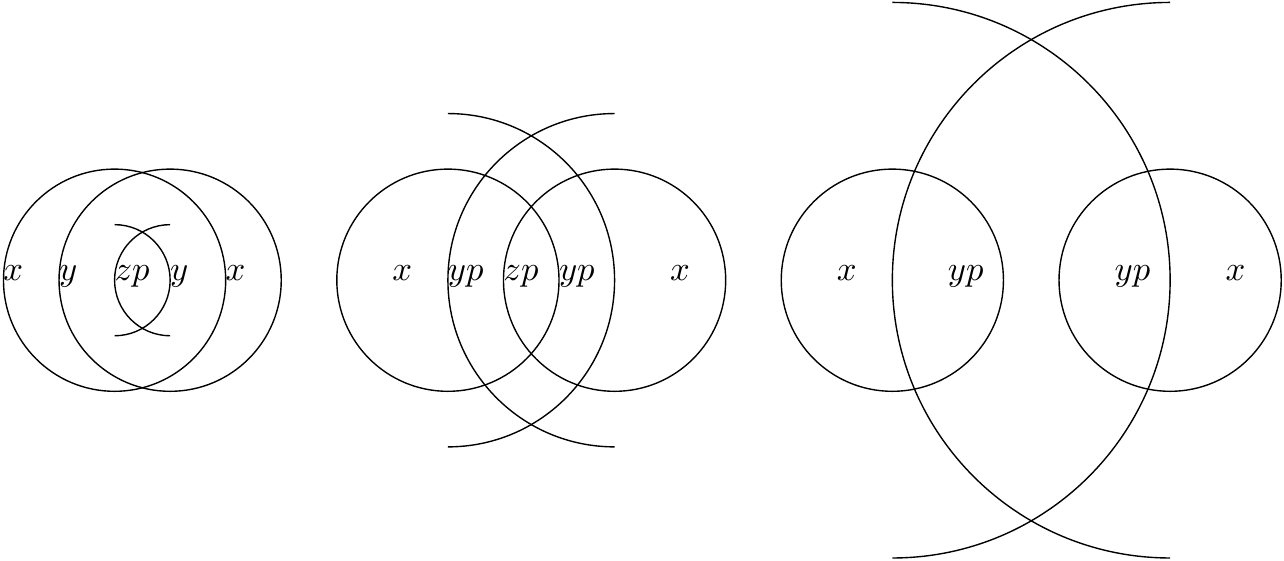}
\caption{\small Stopping an infection in the cases $0<\delta<1,1<\delta<2,\delta>2$}
\label{stopping}
\end{figure}


\noindent{\bf Case 1: $\delta>2$.} In this case, the island occurs with probability at most
\[
q_1=\exp(2r^2f_1(x,y,\delta)+o(r^2))=\exp\left\{\frac{2a\log n}{\pi}f_1(x,y,\delta) +o(\log n)\right\},
\]
where
\[
f_1(x,y,\delta)=(\pi-M(\delta))(x-1-x\log x)+pM(\delta)(y-1-y\log y),
\]
and where we also need
\[
g_1(x,y,\delta)=(\pi-M(\delta))x+pM(\delta)y<\pi\theta
\]
to prevent $\mathbf{u}$ and $\mathbf{v}$ from getting infected. It is easy to see that this configuration is likeliest when $\delta$ is as large
as possible and when $x=y$. A quick calculation with Lagrange multipliers yields the threshold
\[
a\left\{1+p-2\theta+2\theta\log\left(\frac{2\theta}{1+p}\right)\right\}=1.
\]

\noindent {\bf Case 2: $1<\delta<2$.} In this case, the island occurs with probability at most
\[
q_2=\exp(r^2f_2(x,y,z,\delta)+o(r^2))=\exp\left\{\frac{a\log n}{\pi}f_2(x,y,z,\delta)+o(\log n)\right\},
\]
where
\[
f_2(x,y,z,\delta)=2(\pi-M(\delta))(x-1-x\log x)+2p(M(\delta)-L(\delta))(y-1-y\log y)+pL(\delta)(z-1-z\log z),
\]
and where we also need
\[
g_2(x,y,z,\delta)=(\pi-M(\delta))x+p(M(\delta)-L(\delta))y+pL(\delta)z<\pi\theta
\]
to prevent $\mathbf{u}$ and $\mathbf{v}$ from getting infected. A quick calculation with Lagrange multipliers yields that at the optimum $(x,y,z)=(x,x,x^2)$,
and that $x$ and $\delta$ are obtained by solving the equations
\begin{align*}
pL'(\delta)(x-1)+2M'(\delta)(p-1)&=0\\
x(\pi+(p-1)M(\delta)-pL(\delta))+x^2pL(\delta)&=\pi\theta,
\end{align*}
after which we set $q_2=1/n$ to get the bound on the threshold.

\noindent {\bf Case 3: $0<\delta<1$.} In this case, the island occurs with probability at most
\[
q_3=\exp(r^2f_3(x,y,z,\delta)+o(r^2)=\exp\left\{\frac{a\log n}{\pi}f_3(x,y,z,\delta)+o(\log n)\right\},
\]
where
\[
f_3(x,y,z,\delta)=2(\pi-L(\delta))(x-1-x\log x)+(L(\delta)-\delta^2L(1))(y-1-y\log y)+p\delta^2L(1)(z-1-z\log z),
\]
and where we also need
\[
g_3(x,y,z,\delta)=(\pi-L(\delta))x+(L(\delta)-\delta^2L(1))y+p\delta^2L(1)z<\pi\theta
\]
to prevent $\mathbf{u}$ and $\mathbf{v}$ from getting infected. A quick calculation with Lagrange multipliers yields that at the optimum $(x,y,z)=(x,x^2,x^2)$,
and that $x$ and $\delta$ are obtained by solving the equations
\begin{align*}
L'(\delta)(x-1)+2(p-1)(x+1)\delta L(1)&=0\\
x(\pi-L(\delta))+x^2(L(\delta)+\delta^2L(1)(p-1))&=\pi\theta,
\end{align*}
after which we set $q_3=1/n$ to get the bound on the threshold.

\subsection*{\sf Behaviour as $p\to0$ and $p\to1$}

When $p\to 0$, a routine analysis shows that the optimum $\delta$ tends to infinity (so that the threshold in Case 1 serves as the lower bound).
The threshold is thus tangent to the line
\[
a(1-2\theta+2\theta\log(2\theta))=1,
\]
which in turn shows that $\theta_{\mathrm{islands}}(a,0)>0$ for $a>1$, as illustrated in Figure 1.

When $p\to 1$, a routine analysis shows that the optimum $\delta$ tends to zero (so that the threshold in Case 3 serves as the lower bound).
The threshold is thus tangent to the line
\[
a(1-\theta+\theta\log\theta)=1,
\]
which matches the upper bound from Appendix A.
	

\section*{Appendix C: tangency results}
\subsection*{\sf Tangency at $p=0$}

In this subsection, we show that the starting threshold and the local growth threshold $\theta_{\mathrm{local}}$
are tangent at $p=0$. The idea is simple. When $p$ and $\theta=\theta_{\rm start}(p)$ are small, we consider the effect
of lowering the infection threshold from $\theta$ to $\tilde{\theta}=\theta-\delta$, with $\delta=o(\theta)$. Around a point $\mathbf{u}$ that was
newly (i.e., not initially) infected under the higher threshold $\theta$, we now see a small disc $D=B_{\varepsilon r}(\mathbf{u})$ of newly
infected points for some small constant $\varepsilon>0$ (depending on $\delta$). Next, consider a point $\mathbf{v}\in\partial D$. Once again, $\mathbf{v}$ now sees increased infection in
$D$, but also an infection rate of $p$, instead of $\theta$, in the large lune $L=B_r(\mathbf{v})\setminus B_r(x)$. However, since $\theta$
and $p$ are both very small, the vast increase in the infection rate in $D$ more than compensates for the greater area
of the lune $L$, in which the infection rate is only a little lower than that in $B_r(\mathbf{u})$. Consequently, the infection grows
outward from $\mathbf{u}$.

Going into more detail, note that $\theta_{\rm start}=(-a\log p)^{-1}(1+o(1))$ as $p\to 0$. A more careful analysis now yields
the following result.

\begin{proposition}\label{prop: grow1}
	Let $C>\pi^{-2}$, and write $\theta=\theta_{\rm start}(p)$. Then, for sufficiently small $p$, depending on $a$ and $C$, we have
	\[
	\theta-C\theta^2\le \theta_{\mathrm{local}}(p)\le \theta=(-a\log p)^{-1}(1+o(1)).
	\]
	In particular, the starting threshold and local growth threshold are tangent at $p=0$.
\end{proposition}
\begin{proof}
	Write $\tilde{\theta}=\theta-C\theta^2=\theta-\delta$. Suppose that an infection, started at $\mathbf{u}$, has spread to the disc
	$D=B_{\varepsilon r}(\mathbf{u})$. For a point $\mathbf{v}\in\partial D$, the area of the lune $L=B_r(\mathbf{v})\setminus B_r(\mathbf{u})$ is $(2\varepsilon+O(\varepsilon^3))r^2$. 
	Therefore, ignoring second-order terms, the condition for $\mathbf{v}$ to be infected (and the infection
	to spread) under the boostrap percolation model with threshold $\tilde{\theta}$ is
	\[
	\pi\varepsilon^2+(\pi(1-\varepsilon^2)-2\varepsilon)\theta+2\varepsilon p\ge \tilde{\theta}\pi=(\theta-\delta)\pi.
	\]
	Noting that $\theta\gg p$ (since $p\log p\rightarrow 0$ as $p\rightarrow 0$), we may replace this by
	\[
	F(\theta,\varepsilon,\delta)=\pi(1-\theta)\varepsilon^2-2\varepsilon\theta+\delta\pi\ge 0.
	\]
	This holds for all $\varepsilon\ge 0$ as long as
	\[
	(1-\theta)C\theta^2=(1-\theta)\delta\ge \pi^{-2}\theta^2.
	\]
	For sufficiently small $p$, this last inequality is guaranteed by the hypothesis $C>\pi^{-2}$, proving the first inequality
	in the theorem. The remaining inequalities follow from the definitions.
\end{proof}

An analysis of the argument reveals that, with $\delta=\pi^{-2}\theta^2$, a small disc $B_1=B_{\varepsilon_1r}(\mathbf{u})$ is
immediately infected, where $\varepsilon_1=\theta/2\pi$. After that, growth is progressively more difficult, in that the function
$F$ decreases, until the critical radius $\varepsilon_2=2\varepsilon_1=\theta/\pi$, at which point $F=0$. After that, $F$ increases,
and growing proceeds more and more easily.

\subsection*{\sf Tangency at $p=1$}

Next we show that the local growth threshold $\theta_{\mathrm{local}}$ is tangent to the limiting growth threshold
$\theta=\frac{1+p}{2}$ at $p=1$.
Again, the idea is simple. Let us take $p=1-\delta$, so that, along the limiting growth threshold, $\theta=1-\delta/2$.
Since we are well away from the starting threshold, the initial infection rate of $p$ will result in large circular regions
$B_{Kr}(\mathbf{u})$ (where $K=\Theta(1/\delta)$) in which the initial infection density is $\theta=1-\delta/2$, not $p=1-\delta$,
so that every point in $B_{(K-1)r}(\mathbf{u})$ will immediately become infected.

Let us now reduce the infection threshold $\theta$ from $1-\delta/2$ to $1-\delta/2-C\delta^2$, and assume that the infection
has spread to a disc $D=B_{Lr}(\mathbf{u})$, where $L\ge K-1$. Consider a point
$\mathbf{v}\in\partial D$. The point $\mathbf{v}$ will see an infection density of 1 in $A=B_r(\mathbf{v})\cap D$, and a density of at least $p$ in
$B=B_r(\mathbf{v})\setminus D$. Due to the curvature of $D$, the area of $A$ is slightly less than that of $B$, so that the
average infection density in $B_r(\mathbf{v})$ will be $1-\delta/2-C\delta^2$ instead of $\tfrac{1+p}{2}=1-\delta/2$. However,
we have lowered the threshold to $1-\delta/2-C\delta^2$ for precisely this reason, so that $\mathbf{v}$ becomes infected, and
the infection continues to spread.

Making these estimates rigorous is just a matter of bounding the Poisson distribution, as in the following proof.

\begin{proposition}\label{prop: grow2}
	Let $C>(6\sqrt{2}\pi)^{-1}$. Then, if $\delta=1-p$ is sufficiently small (depending on $a$ and $C$), we have
	\[
	\frac{1+p}{2}-C\sqrt{a}\delta^2=\frac{1+p}{2}-C(1-p)^2\le\theta_{\mathrm{local}}(p)\le\frac{1+p}{2}.
	\]
	In particular, the local growth threshold and limiting growth threshold are tangent at $p=1$.
\end{proposition}
\begin{proof}
	
	Let $C$ and $\delta$ be as in the statement of the proposition, and let $\theta=\frac{1+p}{2}-C\sqrt{a}\delta^2$.
	Consider a disc $D=B_{Kr}(\mathbf{u})$ of radius $Kr$, where $K=K(\delta)$ is large but to be determined. With an initial
	infection parameter of $p$, we expect to see
	\[
	p\pi K^2r^2=pa K^2\log n
	\]
	infected points in $D$. If we see instead
	\[
	\theta\pi K^2r^2=\theta aK^2\log n=paK^2\log n(1+\delta/2)(1+o(1))
	\]
	infections, uniformly distributed across $D$, then every point in $D'=B_{(K-1)r}(x)$ will immediately become infected.
	Setting $\rho=1+\delta/2$, the probability $q$ of this occurring is given by
	\[
	q=(1+o(1))e^{paK^2\log n(\rho-1-\rho\log\rho)}=e^{-(1+o(1))pa\delta^2K^2\log n/8}=n^{-(1+o(1))pa\delta^2K^2/8},
	\]
	by Lemma~\ref{lemma: Paul}.
	Thus we should expect to see some fully infected discs of radius $Kr$, where
	\begin{equation}\label{eq: Kone}
	K=\frac{1}{\delta}\sqrt{\frac{8}{a}}(1+o(1)).
	\end{equation}
	
	Next, we show that if $K$ is sufficiently large, the infection will continue to spread
	indefinitely. For $K\ge 1/2$, write $M(K)$ for the area of the lune formed by a disc $B_K$ of radius $K$ and a unit disc
	whose center lies on the perimeter of $B_K$. Exact formulas are given in Appendix B, but asymptotically
	\[
	M(K)=\frac{\pi}{2}-\frac{1}{3K}+O\left(\frac{1}{K^2}\right).
	\]
	Now, if the infection has already spread to all of $B_{Lr}(\mathbf{u})$, where $L\ge K-1$, the condition for it to grow further
	(and indefinitely) is that
	\[
	M(L)+(\pi-M(L))p\ge\pi\theta.
	\]
	Recall that $p=1-\delta$ and $\theta=1-\delta/2-C\sqrt{a}\delta^2$. Using the above approximation for $M(K)$, and ignoring
	second-order terms (so that we may replace $L\ge K-1$ by $L\ge K$, for instance), we can write the condition for
	the infection to spread as
	\[
	\left(\frac12-\frac{1}{3\pi K}\right)+\left(\frac12+\frac{1}{3\pi K}\right)(1-\delta)\ge 1-\frac{\delta}{2}-C\sqrt{a}\delta^2,
	\]
	or
	\begin{equation}\label{eq: Ktwo}
	K\ge\frac{1}{3\pi C\sqrt{a}\delta}.
	\end{equation}
	Combining (\ref{eq: Kone}) and (\ref{eq: Ktwo}) with the hypothesis $C>(6\sqrt{2}\pi)^{-1}$ yields the result.
\end{proof}

\end{document}